\documentclass{amsart}%
\pdfoutput=1
\usepackage{eurosym}
\usepackage{amsmath}
\usepackage{amsfonts}
\usepackage{graphicx}
\usepackage{amsfonts}
\usepackage{mathrsfs}
\usepackage{amssymb}%
\usepackage{xcolor}
\usepackage{exscale}
\usepackage{latexsym}
\usepackage{esint,comment}
\usepackage{tikz}
\usetikzlibrary{decorations.pathreplacing}
\usetikzlibrary{arrows,calc}
\everymath{\displaystyle}

 \numberwithin{equation}{section}

\newtheorem{theorem}{Theorem}[section]
\newtheorem{corollary}[theorem]{Corollary}
\newtheorem{lemma}[theorem]{Lemma}
\newtheorem{proposition}[theorem]{Proposition}

 \newtheorem{defn}[theorem]{Definition}
\newcommand{\bfs}{\boldsymbol}
\newcommand{\bdx}{\mathbf{x}}

\newcommand{\bdz}{{\bf z}}
\newcommand{\bdy}{{\bf y}}
\newcommand{\mcD}{{\mathscr D}}
\usepackage{eurosym}
\usepackage{soul}
\usepackage[top=1in, bottom=1in, left=1.25in, right=1.25in]{geometry}
\begin{document}

%-------------------------------------------------------------------------
% editorial commands: to be inserted by the editorial office
%
%\firstpage{1} \volume{228} \Copyrightyear{2004} \DOI{003-0001}
%
%
%\seriesextra{Just an add-on}
%\seriesextraline{This is the Concrete Title of this Book\br H.E. R and S.T.C. W, Eds.}
%
% for journals:
%
%\firstpage{1}
%\issuenumber{1}
%\Volumeandyear{1 (2004)}
%\Copyrightyear{2004}
%\DOI{003-xxxx-y}
%\Signet
%\commby{inhouse}
%\submitted{March 14, 2003}
%\received{March 16, 2000}
%\revised{June 1, 2000}
%\accepted{July 22, 2000}
%
%
%
%---------------------------------------------------------------------------
%Insert here the title, affiliations and abstract:
%

\title[$L^{p}$ compactness criteria]{$L^{p}$ compactness criteria with an application to variational convergence of some nonlocal energy functionals}
\author{Qiang Du}
\address{
Department of Applied Physics and Applied Mathematics, 
  and the Data Science Institute,
Columbia University, New York, NY 10027}
\thanks{}
\email{qd2125@columbia.edu}
\author{Tadele Mengesha}
%\email{mengesha@utk.edu}
\address{
Department of Mathematics,
University of Tennessee Knoxville, TN}\thanks{
}
\email{mengesha@utk.edu}
\author{Xiaochuan Tian}
\address{Department of Mathematics, University of California, San Diego, CA}
\email{xctian@ucsd.edu}



\begin{abstract}
Motivated by some variational problems from a nonlocal model of mechanics, this work presents a set of sufficient conditions that guarantee a compact inclusion in the function space of $L^{p}$ vector fields defined on a domain $\Omega$ that is either a bounded domain in $\mathbb{R}^{d}$
or  $\mathbb{R}^{d}$ itself. The criteria are nonlocal and are given with respect to nonlocal interaction kernels that may not be necessarily radially symmetric. Moreover, these criteria for vector fields are also different from those given for 
 scalar fields in that the conditions are based on nonlocal interactions involving only parts of the components of the vector fields.  The $L^{p}$ compactness criteria {are} utilized in demonstrating {the} convergence of minimizers of parameterized {nonlocal} energy functionals.   

\end{abstract}

\maketitle

 \section{Introduction and main results}
 \subsection{Motivation}
{The present work is motivated by the study of 
nonlocal peridynamics models initially proposed by Silling 
in \cite{Silling}. In particular, the}
state-based peridynamics model given %proposed %by Silling and
in \cite{Silling2007, Silling, Silling2010} postulates that the total strain energy  for constitutively 
linear, isotropic solid undergoing deformation is given by 
\begin{equation}\label{quad-energy}
\begin{aligned}
W_{\rho}({\bf u})&=\beta\,\int_{\Omega}\left(\mathfrak{D}_{\rho}({\bf u})(\bdx)\right)^2d\bdx \\
&+\alpha \int_{\Omega}\int_{\Omega} \rho(\bdx'-\bdx) \left( \mcD({\bf u})(\bdx,\bdx') -\frac{1}{d}\,\mathfrak{D}_{\rho}({\bf u})(\bdx)\right)^2d\bdx'\,d\bdx
\end{aligned}
\end{equation}
 where $\Omega\subset \mathbb{R}^d$ is a bounded domain occupied by the solid material, the kernel
 $\rho(|{\bfs \xi}|)$ 
  is a nonnegative locally integrable and radial weight function that measures the  interaction strength between  material particles at position $\bdx$ and $\bdx'$, ${\bf u}$ is a displacement field,
 $\mcD({\bf u})$ is a rescaled nonlocal operator on ${\bf u}$ defined by \cite{Du-NonlocalCalculus}
 \begin{equation}
 \label{mcd}
\mcD({\bf u})(\bdx,\bdx')=\frac{({\bf u}(\bdx')-{\bf u}(\bdx))}{|\bdx'-\bdx|}\cdot
 \frac{(\bdx'-\bdx)}{|\bdx'-\bdx|}=\frac{({\bf u}(\bdx')-{\bf u}(\bdx))^{\rm T}(\bdx'-\bdx)}{|\bdx'-\bdx|^2}
 \,,
 \end{equation}
 representing a (unit-less) linearized nonlocal strain \cite{Silling2010}
 and the operator $\mathfrak{D}_{\rho}$
 is a nonlocal linear operator (a weighted version
  of $\mcD$ \cite{Du-NonlocalCalculus,Du-Navier1}), called `nonlocal divergence', which is defined as
\begin{equation}\label{nonlocal-Div}
\mathfrak{D}_{\rho}({\bf u})(\bdx):= p.v. \int_{\Omega} \rho(\bdx'-\bdx) \mcD({\bf u})(\bdx,\bdx')d\bdx'
\end{equation}
which  is a means of incorporating the effect of the collective deformation of a neighborhood of $\bdx$ into the model. 
  The positive constants $\alpha$ and $\beta$  are proportional to the shear and bulk moduli of the material, respectively.  The quadratic energy in \eqref{quad-energy} is a generalization of the bond-based model that was introduced in \cite{Silling} and studied in \cite{Aksoylu-Mengesha,Du-Zhou2011,Emmrich-Weckner2007,Mengesha-Du,Du-Zhou2010} that takes in to account the linearized strain due to the dilatation
and the deviatoric portions of the deformation. Mathematical analysis of  linearized peridynamic models have been extensively studied in \cite{Du-Review1,Du-NonlocalCalculus,Du-Navier1, Du-Zhou2011, Emmrich-Weckner2007, Mengesha-Du, Mengesha-Du-Royal, Du-Zhou2010} along
with results geared towards nonlinear models in 
 \cite{EP13,Lipton14,Mengesha-Du-non,  
Carlos, Carlos2,coclite2018}.

For  $\rho \in L^{1}_{loc}$, it is not difficult to show (see Lemma \ref{W-integrand} below) that the energy space associated with the energy functional $W_\rho$, $\{{\bf u}\in L^{2}(\Omega;\mathbb{R}^{d}): W_{\rho}({\bf u}) < \infty \}$, is precisely 
 \begin{equation}\label{energy-space-p=2}
 \mathcal{S}_{\rho, 2}(\Omega) = \left\{{\bf u}\in L^{2}(\Omega; \mathbb{R}^{d}):  |{\bf u}|^{2}_{\mathcal{S}_{\rho   {, 2}}} < \infty\right\}\,,
 \end{equation}
 where the seminorm  $|{\bf u}|_{\mathcal{S}_{\rho,2}}$ is defined by
\[ |{\bf u}|^{2}_{\mathcal{S}_{\rho,2}} := \int_{\Omega } \int_{\Omega } \rho(\bdy - \bdx)\left|\frac{({\bf u}(\bdy) - {\bf u} ({\bdx}))}{|\bdy-\bdx|} \cdot \frac{(\bdy -\bdx)}{|\bdy - \bdx|}\right|^{2}d\bdy d\bdx. 
\]
Notice that $|{\bf u}|_{\mathcal{S}_{\rho, 2}} = 0, $ if and only if ${\bf u}$ is an infinitesimal rigid vector field. We denote  the
class of infinitesimal rigid displacements by
\[
\mathcal{R}:= \{{\bf u}: {\bf u}({\bf x}) = \mathbb{B} {\bf x} + {\bf v}, \mathbb{B}^{T} = -\mathbb{B}, {\bf v} \in \mathbb{R}^{d}\}.
\]
It has been shown in \cite{Mengesha-Du, Mengesha-Du-non} that $\mathcal{S}_{\rho, 2}(\Omega) $ with the natural norm
\[
\|{\bf u}\|_{\mathcal{S}_{\rho,2}} = {(\|{\bf u}\|^2_{L^2} + |{\bf u}|^2_{\mathcal{S}_{\rho, 2}})^{1/2}}
\]
is a separable Hilbert space.  In the event that $\rho({\bfs \xi})|{\bfs \xi}|^{-2} \in L^{1}_{loc}(\mathbb{R}^{d})$, then the space $\mathcal{S}_{\rho,2}(\Omega)$ coincides with $L^{2}(\Omega, \mathbb{R}^{d})$. Otherwise, $\mathcal{S}_{\rho,2}(\Omega)$ is a proper subset of $L^{2}(\Omega, \mathbb{R}^{d})$     {that is}, nevertheless, big enough to contain $W^{1,2}(\Omega;\mathbb{R}^{d})$
and there exists a constant $C=  C(d, 2, \Omega)>0$   such that 
 \[
 |{\bf u}|^{2}_{\mathcal{S}_{\rho,2}} \leq C  \|\text{Sym}(\nabla {\bf u})\|^{2}_{L^{2}} \|\rho\|_{L^{1}(\mathbb{R})} \quad\quad\forall {\bf u}\in W^{1, 2}(\Omega; \mathbb{R}^{d})
 \]
 where $\text{Sym}(\nabla {\bf u}) = \frac{1}{2} (\nabla {\bf u} + \nabla {\bf u} ^{T})$ is the symmetric part of the gradient $\nabla {\bf u}$.

Under the additional assumption that  $\rho$ is positive in a small neighborhood of the origin,  
it is shown in \cite[Theorem 1]{Mengesha-Du}, via     {an} application of Lax-Milgram, that for any applied load ${\bf f}\in L^{2}(\Omega;\mathbb{R}^{d})$, the potential energy 
\begin{equation}\label{pot-p=2}
E_\rho({\bf u}) = W_{\rho} ({\bf u})  -\int_{\Omega} {\bf f}\cdot {\bf u} dx
\end{equation}
has a minimizer over any weakly closed subset  $V$ of $ \mathcal{S}_{\rho, 2}(\Omega) $ such that $V\cap \mathcal{R} = \{{\bfs 0}\}$. See also  \cite[Theorem 1.1]{Mengesha-Du-non} for the more general convex energies of $p$-growth. 

The analysis of the convergence of variational problems of the type in \eqref{pot-p=2} associated with a sequence of parameterized kernels has garnered a lot of attention in recent years.   Namely, if we have a sequence of locally integrable radial kernels $\rho_{n}$, how do the associated potential energies ${E_{\rho_n}}$, as well as their minimizers behave as $n\to \infty$? Clearly, this will depend first on the behavior of  the convergence properties of the sequence of kernels. In fact, it is shown in \cite{Mengesha-Du-non}  that if the sequence of $L^{1}$ kernels $\{\rho_n\}$ converge in {\em the sense of measures} to a measure with atomic mass at $0$ (Dirac-measure at $0$) and for each $n$, $r^{-2}\rho_{n}(r)$ is nonincreasing, then the sequence $\{E_{\rho_n}\}$ variationally converges to the classical Navier-Lam\'e potential energy $E_0$ given by 
\[
E_{0}({\bf u}) =\mu \int_{\Omega}|\text{Sym}(\nabla {\bf u})|^2 dx + \frac{\lambda}{2}\int_{\Omega} (\text{div} ({\bf u}))^2 dx - \int_{\Omega} {\bf f}\cdot {\bf u} dx,
\]
 where $\mu$ and $\lambda$ are constants that can be expressed in terms of $\alpha$ and $\beta$. 
 This is what is called nonlocal-to-local convergence and the result is used as a rigorous justification that state-based peridynamics modeling recovers the classical linearized elasticity models in the event of vanishing nonlocality. 
 
 In this paper, we consider another type of convergence of sequence of kernels and study the behavior of the associated energy functionals, which leads to nonlocal-to-nonlocal convergence. More specifically, suppose we are given a nonnegative kernel $\rho\in L^{1}_{loc}(\mathbb{R}^d)$ with the property that 
 \begin{equation}\label{rho-inc-rad1-p=2}
 \text{$\rho $ is {\em radial}, $\rho({\bfs \xi})>0$ \text{for ${\bfs \xi}$ is close to ${\bfs 0}$}, $|{\bfs \xi}|^{-2}\rho(|{\bfs \xi}|)$ is {\em nonincreasing} in $|{\bfs \xi}|$,}
 \end{equation}
 and 
 \begin{equation}\label{suff-cond-p=2}
 \lim_{\delta \to 0}\delta^{2} \left(\displaystyle \int_{B_{\delta}}  \rho({\bfs \xi})d{\bfs \xi} \right)^{-1} = 0. 
 \end{equation}
and consider a sequence of nonnegative, radial kernels $\{\rho_{n}\}_{n\in \mathbb{N}}$ each satisfying  
\eqref{rho-inc-rad1-p=2} and that 
 \begin{equation}\label{condition-on-seq-ker-p=2}
  \rho_n\leq \rho\quad \text{a.e. and } \rho_{n}\to \rho \quad\text{a.e. in $\mathbb{R}^{d}$}. 
 \end{equation}
 It then follows that $\rho_n \to \rho$ strongly in $L^{1}_{loc}(\mathbb{R}^d)$ as $n\to \infty$. 
We will establish a clear connection between the sequence of energies $\{E_{\rho_n}\}$ and $E_{\rho}$.  Most importantly, we will show that minimizers of the energies $E_{\rho_n}$ over an admissible class will converge to a minimizer of $E_{\rho}$ over the same admissible class.   
The notion of variational convergence we use is $\Gamma$-convergence (see \cite{DalMaso}) which we define below. The advantage of $\Gamma$-convergence is that under the additional assumption of equicoercivity of the functionals it implies the convergence of minimizers as well \cite[Theorem 7.8 and Corollary 7.202]{DalMaso}.   
 \begin{defn}\label{defn-gamma}
Suppose that $\overline{E}_{n}: L^{2}(\Omega;\mathbb{R}^{d})\to \mathbb{R}\cup \{\infty\}$,  $\forall 1\leq n\leq \infty$.  We say that the sequence $\overline{E}_{n}$ $\Gamma$- converges to $\overline{E}_{\infty}$ in the  $L^{2}$-topology if and only if
\begin{itemize}
\item [a)] for every sequence $\{{\bf u}_{n}\}\in L^{2}(\Omega;\mathbb{R}^{d})$ with ${\bf u}_{n}\to {\bf u}$ in $L^{2}(\Omega;\mathbb{R}^{d}),$ as $n\to\infty,$ we have $\overline{E}_{\infty}({\bf u})\leq \liminf_{n\to \infty}\overline{E}_{n}({\bf u}_{n})$,
    \item [b)] and for every ${\bf u}\in L^{2}(\Omega;\mathbb{R}^{d})$ there exists a {recovery} sequence ${\bf u}_{n}\to {\bf u}$ in $L^{2}(\Omega;\mathbb{R}^{d}),$ such that $\overline{E}_{\infty}({\bf u})=\lim_{n\to \infty}\overline{E}_{n}({\bf u}_{n}).$ 
\end{itemize}
\end{defn}

The following is one of the main results of the paper on the variational limit of the nonlocal functionals $\{\overline{E}_{n}\}$. 

 \begin{theorem} \label{thm-G-convergence-intro} 
Suppose $\rho$ and $\{\rho_n\}$ satisfies \eqref{rho-inc-rad1-p=2}, \eqref{suff-cond-p=2}, and \eqref{condition-on-seq-ker-p=2}. 
The sequence of functionals $\overline{E}_{n}$ $\Gamma-$converges in the strong $L^{2}(\Omega;\mathbb{R}^{d})$ topology to the functional $\overline{E}_{\rho}$.
where 
the extended functionals $\{\overline{E}_{n}({\bf u})\}_{n\leq \infty}$  are defined as
\begin{equation}
\label{extended}
\overline{E}_{n}({\bf u}) =\left\{\begin{array}{l}
E_{\rho_n}({\bf u}), \quad\text{if ${\bf u}\in \mathcal{S}_{\rho_n, 2}(\Omega)$,}\\
\infty,\quad\text{ if ${\bf u}\in L^{2}(\Omega;\mathbb{R}^{d})\setminus \mathcal{S}_{\rho_n, 2}(\Omega) $},
\end{array}\right.
\end{equation}
where $\mathcal{S}_{\rho_n, 2}(\Omega)$ and $\{E_{\rho_{n}}\}_{n<\infty}$ are defined as before in  \eqref{energy-space-p=2} and \eqref{pot-p=2}, respectively, where $\rho$ is replaced by $\rho_n$. The extended functional $\overline{E}_{\rho}$ is similarly defined. 
Moreover, if $V$ is a weakly
closed subset of $L^{2}(\Omega;\mathbb{R}^{d})$ such that $V\cap \mathcal{R} = \{\bfs 0\}$, and  for each n, ${\bf u}_{n}$ minimizes $E_{\rho_n}$ over $V\cap \mathcal{S}_{\rho_n, 2}(\Omega)$, then 
the sequence $\{{\bf u}_n\}$ is precompact in $L^{2}(\Omega;\mathbb{R}^{d})$ with any limit point belonging to $\mathcal{S}_{\rho,2}(\Omega)$ and minimizing $E_{\rho}$ over $V\cap \mathcal{S}_{\rho,2}(\Omega)$. 
 \end{theorem}

Although the discussion above is focused on the case of quadratic peridynamic energies for ease of explaining the main ideas, the result can naturally be extended to {\em small strain nonlocal nonlinear peridynamic models} with $p$-growth, for $p\geq 2$, that have been introduced in \cite{Silling2007} and whose variational analysis investigated in \cite{Mengesha-Du-non}.

We will prove Theorem \ref{thm-G-convergence-intro} in the sections that follow. But we would like to highlight that this result has an important implication in the numerical approximation of minimizers of  $E_{\rho}$ over an admissible class. 
Indeed, compactness results have been quite useful for analyzing numerical approximations of nonlocal problems in various contexts such as \cite{Asy-compatible,Tian-Du,coclite2022}. In the context discussed in this work, let
us take for an example that $\rho({\bfs \xi})  = \frac{1}{|{\bfs \xi}|^{d + 2(s-1)}}$ for $s\in (0,1)$. This kernel satisfies \eqref{rho-inc-rad1-p=2} and \eqref{suff-cond-p=2}. It is also clear that $\rho({\bfs \xi})|{\bfs \xi}|^{-2}$ is not integrable on any bounded domain containing ${\bfs 0}$.  In the event $\Omega$ has a smooth boundary, the energy space $\mathcal{S}_{\rho,2}(\Omega)$ coincides with  the fractional Sobolev space $H^{s}(\Omega;\mathbb{R}^d)$ see \cite{Mengesha-Scott1, Mengesha-Scott2}. In particular, if ${1/2} < s < 1$, then all functions in $\mathcal{S}_{\rho,2}(\Omega)$ have continuous representative. Now, if  
$V\subset\mathcal{S}_{\rho,2}(\Omega)$ is a weakly
closed subset of $L^{2}(\Omega;\mathbb{R}^{d})$ such that $V\cap \mathcal{R} = \{\bfs 0\}$, a minimizer of $E_{\rho}$ over $V$ exists (and will be in $H^{s}(\Omega;\mathbb{R^d})$).  The analysis of the existence and uniqueness of the minimizer ${\bf u}$ of this quadratic energy can also, equivalently, be found by solving the corresponding Euler-Lagrange equation. The latter gives us a way of numerically solving the solution by writing it first in the weak form and then  {applying} the Galerkin approach of choosing a finite-dimensional subspace $\mathcal{M} \subset V$ to solve for a projected
solution of ${\bf u}$ on $\mathcal{M}$. Notice that for $s\in (0, 1/2)$ the finite-dimensional subspace $\mathcal{M}$ can contain discontinuous functions, while for $s\in (1/2, 1)$, all the elements of $\mathcal{M}$ must be continuous in order for $\mathcal{M}$ to be conforming, that is, for $\mathcal{M}\subset \mathcal{S}_{\rho,2}(\Omega)$ .  In the    {latter} case, if one wants to employ the advantageous discontinuous Galerkin approximation, which is now nonconforming, one needs to find 
   {an} effective way to implement it to the model problem.  The result in Theorem \ref{thm-G-convergence-intro} will allow us to develop approximation schemes by solving a sequence of Euler-Lagrange equations of modified energies.  To demonstrate this, define the sequence of kernels
\[
\rho_{n}({\bfs \xi}) 
= \left\{\begin{aligned}\rho({\bfs \xi}) & &\text{if $\rho({\bfs \xi})\leq n\,|{\bfs \xi}|^{2}$}\\
 n \,|{\bfs \xi}|^{2} &&\text{if $\rho({\bfs \xi}) \geq n \,|{\bfs \xi}|^{2}. $}\end{aligned}
\right.
\]
It is not difficult to check that, for each $n$, $\rho_{n}$  satisfies \eqref{rho-inc-rad1-p=2}, \eqref{condition-on-seq-ker-p=2}, and that the functions $\rho_{n}({\bfs \xi})|{\bfs \xi}|^{-2} \in L^{1}_{loc}(\mathbb{R}^{n})$ are just truncations of the fractional kernel $|{\bfs \xi}|^{-d-2s}$ (at level $n$).  As discussed before, the energy space associated with $E_{\rho_n}$ will coincide with $L^{2}(\Omega;\mathbb{R}^{d})$ and a unique minimizer ${\bf u}_n \in V\cap \mathcal{S}_{n}$ of $E_{\rho_n}$ exists.  Since the admissible space is a subspace of  $L^{2}(\Omega;\mathbb{R}^{d})$ that avoids nontrivial infinitesimal rigid displacements, we may use discontinuous finite element spaces, %say
    {denoted by} $\mathcal{M}_{n, h}$, for     {the} standard conforming Galerkin approximation to   
the solution of the Euler-Lagrange equation associated with the energy $E_{\rho_n}.$ This, in turn, can be viewed as a     {nonconforming} discontinuous Galerkin scheme to numerically solve the original Euler-Lagrange equations when the discretization parameter $h$ goes to zero and at the same time the truncation level $n$ goes to infinity.  Intuitively, for large $n$, by Theorem \ref{thm-G-convergence-intro}, ${\bf u}$ is approximated by ${\bf u}_n$ (in the $L^{2}$ norm), and then ${\bf u}_n$ will be numerically approximated by ${\bf u}_{n, h}\in \mathcal{M}_{n, h}$. The proper convergence analysis of this nonconforming numerical scheme as $h\to 0$ and $n\to \infty$ simultaneously has been carried out in \cite{Tian-Du} in the special case of scalar nonlocal problems when the subspace $V$ is the set of 
    {scalar-valued functions}
${\bf u}\in L^{2}(\Omega;\mathbb{R}^{d})$ such that ${\bf u} $ vanishes outside of a fixed set $\Omega'$ which is compactly contained in $\Omega$. The  analysis in \cite{Tian-Du}     {makes} use of the framework of asymptotically compatible schemes for parameter-dependent problems first developed in \cite{Asy-compatible} and the vanishing of the functions in the admissible class around the volumetric-boundary  $\Omega\setminus \Omega'$ is crucial for employing certain compactness arguments.  To extend the convergence analysis in \cite{Tian-Du} to  the case of a system of strongly coupled nonlocal equations,  
the variations problems associated with \eqref{pot-p=2},
solved over any admissible set that does not include infinitesimal rigid vector fields, Theorem \ref{thm-G-convergence-intro} as well as some of the compactness results we prove in this paper will be crucial. Such analysis on nonconforming discontinuous Galerkin numerical schemes 
{to  
systems of nonlocal equations under discussion will be carried out in a} future work.  

 Although it is beyond the scope of this work, in passing, we would like to note that this way of developing a nonconforming numerical scheme is also applicable to fractional PDEs where singular kernels are more common
\cite{Bonito18,TDG2016}. The idea is the same where we use less singular kernels with truncation both at origin and at infinity to do approximation of fractional PDEs. In this case, sequential compactness of nonlocal spaces  associated with the truncated fractional kernel together with the compact embedding of fractional Sobolev spaces in $L^{p}$ can be used to carry out the analysis of the resulting asymptotically compatible schemes {\cite{Asy-compatible,TDG2016}}. 
 
\subsection{Other main results}
\subsubsection{$L^{p}$ Compactness}
The proof of  Theorem \ref{thm-G-convergence-intro} fundamentally depends on some structural properties of the nonlocal space $\mathcal{S}_{\rho, 2}(\Omega)$,     {chief} among them are compact embedding into $L^{2}(\Omega;\mathbb{R}^d)$ and a Poincar\'e-type inequality, which we will establish in this paper. In fact, these properties remain true even for the spaces $\mathcal{S}_{\rho, p}(\Omega)$, where for $1\leq p<\infty$, 
\[
\mathcal{S}_{\rho, 2}(\Omega) = \{{\bf u}\in L^{p}(\Omega;\mathbb{R}^d): |{\bf u}|_{\mathcal{S}_{\rho,p}} < \infty\}     {,} %. 
\]
    {and} $|{\bf u}|_{\mathcal{S}_{\rho,p}} ^{p} = \int_{\Omega } \int_{\Omega } \rho(\bdy - \bdx)\left|\frac{({\bf u}(\bdy) - {\bf u} ({\bdx}))}{|\bdy-\bdx|} \cdot \frac{(\bdy -\bdx)}{|\bdy - \bdx|}\right|^{p}d\bdy d\bdx$ gives a semi norm.  
It is shown in \cite{Mengesha,Mengesha-Du-non} that, for any $1\leq p < \infty$, 
 $ \mathcal{S}_{\rho, p}(\Omega)$ is a separable Banach space with {the} norm 
 $$\|{\bf u}\|_{\mathcal{S}_{\rho,p}} =\left(\|{\bf u}\|_{L^{p}}^{p} +  |{\bf u}|^{p}_{\mathcal{S}_{\rho,p}}\right)^{1/p},$$
    {and is} reflexive if $1 < p < \infty$ and  a Hilbert space for 
 $p =2$.   If $|{\bfs \xi}|^{-p}\rho({\bfs \xi}) \in L_{loc}^{1}(\mathbb{R}^{d})$, then 
 a simple calculation shows that
  $ \mathcal{S}_{\rho, p}(\Omega) = L^{p}(\Omega;\mathbb{R}^{d})$.   On the other hand, in the case where $|{\bfs \xi}|^{-p}\rho({\bfs \xi}) \notin L^{1}_{loc}(\mathbb{R}^{d})$, $ \mathcal{S}_{\rho, p}(\Omega)$ is a proper subset of $L^{p}(\Omega;\mathbb{R}^{d})$.   Under some extra assumptions on the kernel $\rho$,  the space is known to support a Poincar\'e-Korn type inequality over subsets that have trivial intersections with $\mathcal{R}$. 
  These functional analytic properties of the nonlocal space can
  be used to demonstrate    {the} well-posedness of some nonlocal variational problems using 
  the direct method of calculus of variations, see \cite{Mengesha-Du-non} for more discussions.

 As in the case of $p=2$, we assume that for a given $1\leq p<\infty,$ 
 \begin{equation}\label{rho-inc-rad}
\text{$\rho $ is {\em radial}, $\rho(r)>0$ for $r$ is close to $0$},  \text{and $r^{-p}\rho(r)$ is {\em nonincreasing} in $r$,}
 \end{equation} and 
  \begin{equation}\label{suff-cond}
 \lim_{\delta \to 0}{\delta^{p} }\left({\displaystyle \int_{B_{\delta}({\bfs 0}) } \rho({\bfs \xi})d{\bfs \xi} }\right)^{-1}= 0. 
 \end{equation}
We now state the compactness result whose proof is one of the main objectives of the present work.  
  \begin{theorem}[$L^{p}$ compactness]\label{main-compactness}
 Let $1\leq p<\infty$  and let $\rho \in L^{1}_{loc}(\mathbb{R}^{d})$ be nonnegative and satisfying \eqref{rho-inc-rad} and \eqref{suff-cond}.   
 Suppose also that $\Omega\subset \mathbb{R}^{d}$ is a domain with Lipschitz boundary. 
 Then $\mathcal{S}_{\rho,p}(\Omega)$ is compactly embedded in $L^{p}(\Omega;\mathbb{R}^{d})$. That is,  any bounded sequence $\{{\bf u}_{n}\}$ in $\mathcal{S}_{\rho,p}(\Omega)$  is precompact in $L^{p}(\Omega;\mathbb{R}^{d})$.  Moreover, any limit point  is in $\mathcal{S}_{\rho,p}(\Omega)$. 
 \end{theorem}    
    {The condition given by} \eqref{suff-cond} requires $\rho$ to have     {an} adequate singularity near ${\bfs 0}$.  A straightforward calculation shows that the kernels satisfying \eqref{suff-cond} include 
 $\rho({\bfs \xi}) = |{\bfs \xi}|^{-(d + p(s-1))}$, for any $p\in [1, \infty)$, and any 
 $s\in (0,1)$, and $\rho({\bfs \xi}) = -|{\bfs \xi}|^{p-d} \ln(|{\bfs \xi}|)$.  
 It is no surprise that \eqref{suff-cond}  is violated if $ |{\bfs \xi}|^{-p}\rho({\bfs \xi})$ is a locally  integrable function (and therefore, $\mathcal{S}_{\rho, p}(\Omega) = L^{p}(\Omega;\mathbb{R}^{d})$), and  in fact, in this case
 \[
 \liminf_{\delta \to 0}{\delta^{p} }\left({\displaystyle \int_{B_{\delta}({\bfs 0}) } \rho({\bfs \xi})d{\bfs \xi} }\right)^{-1} = \infty,\]
     {see \cite{Mengesha-Du-non}}. 
 It is not clear whether condition \eqref{suff-cond} is necessary for compact embedding even for the class of kernels that are radial and nonincreasing.   There are radial kernels   with the property that $|{\bfs \xi}|^{-p}\rho({\bfs \xi})$ is (locally) nonintegrable, and 
 $$
 \lim_{\delta \to 0}{\delta^{p} }\left({\displaystyle \int_{B_{\delta} ({\bfs 0})} \rho({\bfs \xi})d{\bfs \xi} } \right)^{-1}= c_{0} > 0
 $$
  for which we do not know whether there is a compact embedding $\mathcal{S}_{\rho, p}(\Omega)$ into $L^{p}(\Omega;\mathbb{R}^{d})$. One such kernel is $\rho({\bfs \xi}) = |{\bfs \xi}|^{p-d}$. Nevertheless, we can prove that the associated space $\mathcal{S}_{\rho,p}(\Omega)$ is compact in the $L^{p}_{loc}$ topology, a result which we will state and prove in the appendix.

\subsubsection{Compactness criteria that involve     {a} sequence of kernels}
 For scalar fields,  compactness results     {like those stated above} are commonplace for spaces corresponding to special kernels such as the standard fractional Sobolev spaces.  
 In \cite[Lemma 2.2]{Mengesha-Du-AIMS}, for more general radial and monotone decreasing kernels $\rho$, condition \eqref{suff-cond} is shown to be sufficient for the compact embedding of the space 
 \[
 \left\{f\in L^{2}(\Omega): \int_{\Omega}\int_{\Omega}\rho(\bdy-\bdx)\frac{|   {f(\bdy) - f(\bdx)}|^{2}}{|   {\bdy-\bdx}|^{2}} < \infty\right\}\; 
 \]
 in $L^{2}(\Omega)$.
  The statement is certainly true for any $1\leq p < \infty.$ The proof of \cite[Lemma 2.2]{Mengesha-Du-AIMS}  actually relies on and modifies the argument used to prove another type of compactness result by Bourgain, Brezis and Mironescu in \cite[Theorem 4]{anotherlook} that applies criteria involving a sequence of kernels. The argument of \cite{anotherlook} uses extensions of functions to $\mathbb{R}^{d}$ where the monotonicity of $\rho$ is used in an essential way to control the semi-norm of the extended functions by the original semi-norm. That is, let us introduce a sequence of radial functions  $\rho_{n}$ satisfying 
 \begin{equation}\label{cond-rhon}
\forall n \geq 1,\, \rho_{n}\geq 0, \,\,  \int_{\mathbb{R}^{d}} \rho_{n}({\bfs \xi}) d{\bfs \xi} = 1,\,\,\text{and\,\,}
  \lim_{n\to \infty} \int_{|{\bfs \xi}| > r} \rho_{n}({\bfs \xi}) d{\bfs \xi} = 0, \, \text{ $\forall r > 0$}.
 \end{equation}
 Assuming that for each $n$, $\rho_{n}$ is nonincreasing, and 
  if 
 \begin{equation}\label{**}
 \sup_{n\geq 1} \int_{\Omega}\int_{\Omega}\rho_{n}(\bdy-\bdx)\frac{|f_{n}(\bdy) - f_{n}(\bdx)|^{p}}{|\bdy-\bdx|^{p}} < \infty,  \end{equation}
 then $\{f_n\}$ is precompact in $L^{p}(\Omega)$, which is the result of
 \cite[Theorem 4]{anotherlook} obtained 
  by showing that \eqref{**} makes it possible to apply a variant of the Riesz-Fr\'echet-Kolomogorov theorem \cite{brezis-Functional}.  In \cite[Lemma 2.2]{Mengesha-Du-AIMS}, for a  fixed $\rho$, the condition \eqref{suff-cond} 
  is used to replace the role played by the condition \eqref{cond-rhon}. 
 In \cite[Theorem 1.2]{Ponce2003},  the same result as in \cite[Theorem 4]{anotherlook} was proved by dropping the monotonicity assumption on $\rho_{n}$ for $d\geq 2$.  
 In addition, the proof in \cite{Ponce2003}
 avoids the extension of functions to $\mathbb{R}^{d}$ but rather shows that the bulk of the mass of 
each ${f_{n}}$, that is $\int_{\Omega} |f_n|^{p}$,   comes from the interior and quantifies the contribution near the boundary. As a consequence, if \eqref{**} holds, then as $n\to \infty$ there is no mass concentration or leak at the boundary, two main causes of failure of compactness.  The compactness results were applied to establish some variational convergence results in \cite{Ponce2004}. 
 Clearly if one merely replaces scalar functions in \eqref{**} by vector fields,  both compactness results \cite[Theorem 4]{anotherlook} and  \cite[Theorem 1.2]{Ponce2003} will remain true.  It turns out the results will remain valid for vector fields even under a weaker assumption.  Indeed, following the argument \cite[Theorem 4]{anotherlook} and under the monotonicity assumption that for $n$, $\rho_{n}$ is nonincreasing, it was proved  in \cite[Theorem 5.1]{Mengesha} that
 if ${\bf u}_{n}$ is a bounded sequence of vector fields satisfying 
 \begin{equation}\label{cond-un}
 \sup_{n\geq 1} \int_{\Omega}\int_{\Omega} \rho_{n}(\bdy - \bdx)\left|\mathscr{D}({\bf u}_{n})(\bdx, \bdy)\right|^{p}d\bdy d\bdx < \infty
 \end{equation}
 then $\{{\bf u}_{n}\}$ precompact in the $L^{p}_{loc}(\Omega;\mathbb{R}^{d})$ topology with any limit point being in $W^{1, p}(\Omega;\mathbb{R}^{d})$ when $1< p < \infty$, and in $BD(\Omega)$ when $p=1$.  Here, $BD(\Omega)$ is the space of functions with bounded deformation.  Later, again under the monotonicity assumption on $\rho_{n}$, but using the argument of \cite[Theorem 1.2]{Ponce2003} instead, it was proved in \cite[Proposition 4.2]{Mengesha-Du-non} that in fact, \eqref{cond-un} implies that $\{{\bf u}_{n}\}$ is precompact in the $L^{p}(\Omega;\mathbb{R}^{d})$ topology.  In this paper, we will prove a similar result relaxing the requirement that $\rho_{n}$ is a Dirac-Delta sequence.  

\begin{theorem}\label{compactness-sequence}
   Let $\rho \in L^{1}_{loc}$ satisfy \eqref{rho-inc-rad} and  \eqref{suff-cond}.
For each $n$,  $\rho_{n}$ is radial and $\rho_{n}$ satisfies 
\eqref{rho-inc-rad} and that 
 \[
 \rho_{n}\geq 0, \quad \rho_{n}\rightharpoonup \rho, \quad\text{weakly in $L^{1}_{loc}(\mathbb{R}^{d})$}, \text{and\, $\rho_n \leq c \rho$}
 \]
 for some $c>0.$ Then, if  $\{{\bf u}_{n}\}$ is a bounded sequence in $L^{p}(\Omega; \mathbb{R}^{d})$ such that \eqref{cond-un} holds, then  $\{{\bf u}_{n}\}$ is precompact in $L^{p}(\Omega;\mathbb{R}^{d})$. Moreover, any limit point is in $\mathcal{S}_{\rho,p}(\Omega).$ 
 \end{theorem}
  A natural by-product of Theorem \ref{compactness-sequence} is
 the Poincar\'e-Korn type inequality stated below.
  
 \begin{corollary}[Poincar\'e-Korn type inequality]\label{Poincare}
 Suppose that $1 \leq p < \infty$ and $V$ is a weakly closed subset of $L^{p}(\Omega; \mathbb{R}^{d})$ such that $V\cap\mathcal{R} = \{{\bfs 0}\}$. Let $\rho \in L^{1}_{loc}$ satisfies \eqref{rho-inc-rad} and  \eqref{suff-cond}.
Let  $\rho_{n}$ be a sequence of radial functions, and for each $n$, $\rho_{n}$ satisfies \eqref{rho-inc-rad} and that 
 \[
 \rho_{n}\geq 0, \quad \rho_{n}\rightharpoonup \rho, \quad\text{weakly in $L^{1}_{loc}(\mathbb{R}^{d})$}, \text{and\, $\rho_n \leq c \rho$}
 \] 
 for some $c>0$.  Then there exist constants $C > 0$ and $N\geq 1$ such that
 \begin{equation}\label{seq-poincare}
 \int_{\Omega} |{\bf u}|^{p}d\bdx \leq C \int_{\Omega}\int_{\Omega}  \rho_{n}(\bdy - \bdx)\left|\frac{({\bf u}(\bdy) - {\bf u}({\bdx}))}{|\bdy-\bdx|} \cdot \frac{(\bdy -\bdx)}{|\bdy - \bdx|}\right|^{p}d\bdy d\bdx
 \end{equation}
 for all ${\bf u}\in V\cap L^{p}(\Omega;\mathbb{R}^{d})$ and $n\geq N$.  The constant $C$ depends only on $V, d,p$,  $\rho$, and the Lipschitz character of $\Omega$. 
 \end{corollary}
 We note that the Poincar\'e-Korn-type inequality has been proved for  Dirac-Delta sequence of kernels $\rho_{n}$ \cite{Mengesha-Du,Mengesha-Du-non}. The corollary extends the result to sequence of kernels that weakly converge to a given function $\rho$ satisfying \eqref{rho-inc-rad} and  \eqref{suff-cond}.

 The rest of the paper is devoted to prove the main results and it is organized as follows.  We prove Theorem \ref{thm-G-convergence-intro} in Section \ref{Gamma-conv-section}. 
 Theorem \ref{loc-compactness-S}  %as we all result
     {and} Proposition  \ref{compactness-sequence-loc}     {are}  proved in section \ref{sec2}. The proof of Theorems  \ref{main-compactness} and
\ref{compactness-sequence}
and Corollary \ref{Poincare}
are presented in section \ref{sec3}.
Further discussions are given at the end of the paper.

 \section{Proof of the variational convergence}\label{Gamma-conv-section}
 In this section we will prove the $\Gamma$-convergence of the sequence of energies $E_{\rho_n}$ defined in \eqref{pot-p=2}. 
 The proof relies on a sequence of results on the limiting behavior of functions as well as the action of operators. To that end, we assume that $\rho$ and $\{\rho_n\}$ satisfy \eqref{rho-inc-rad1-p=2}, \eqref{suff-cond-p=2}, and \eqref{condition-on-seq-ker-p=2} throughout this section.  We begin with the convergence properties of the nonlocal divergence operator.

\begin{lemma}\label{weak-conv-nldiv-seq}
Suppose that ${\bf u}_n \to {\bf u}$ strongly in $L^{2}(\Omega;\mathbb{R}^{d})$, ${\bf u}\in \mathcal{S}_{\rho,2}(\Omega)$, and that $\{\mathfrak{D}_{\rho_n} ({\bf u}_n)\}$ is uniformly bounded in $L^{2}(\Omega)$. Then $\mathfrak{D}_{\rho_n}({\bf u}_n) \rightharpoonup \mathfrak{D}_{\rho}({\bf u})$ weakly in $L^{2}(\Omega)$.
\end{lemma}
\begin{proof}
We recall the nonlocal integration by parts formula (\cite{Mengesha-Spector, Mengesha-Du-non}) that for any $v\in W^{1,2}_0(\Omega)$
\[
\int_{\Omega} \mathfrak{D}_{\rho_n}({\bf u}_n) v(\bdx) dx = -\int_{\Omega} \mathcal{G}_{n}(v)(\bdx) \cdot {\bf u}_{n}(\bdx) d\bdx
 \]
 where $ \mathcal{G}_{n}(v)(\bdx) $ is the nonlocal gradient operator 
 \[
  \mathcal{G}_{\rho_n}(v)(\bdx)  = p.v. \int_{ \Omega} \rho_{n}(\bdy-\bdx)\frac{v(\bdy)+v(\bdx)}{|\bdy-\bdx|} \frac{\bdy-\bdx}{|\bdy-\bdx|}d\bdy.
 \]
 Now for $v\in C^{1}_{c}(\Omega)$, we may rewrite the nonlocal gradient as 
 \[
 \mathcal{G}_{\rho_n}(v)(\bdx)=\int_{ \Omega} \rho_{n}(\bdy-\bdx)\frac{v(\bdy)-v(\bdx)}{|\bdy-\bdx|} \frac{\bdy-\bdx}{|\bdy-\bdx|}d\bdy + 2  \int_{ \Omega} \rho_{n}(\bdy-\bdx)\frac{v(\bdx)}{|\bdy-\bdx|} \frac{\bdy-\bdx}{|\bdy-\bdx|}d\bdy.
 \]
  and estimate as \cite[Corollary 2.4]{Mengesha-Spector},
 \[\| \mathcal{G}_{\rho_n}(v)\|_{L^{\infty}} \leq 3 \|\rho_{n}\|_{L^{1}}\|\nabla v\|_{L^{\infty}} \leq  C \|\rho\|_{L^{1}}\|\nabla v\|_{L^{\infty}}.\]
 Also, it is not difficult to show that for all $\bdx\in \Omega$, $ \mathcal{G}_{\rho_n}(v)(\bdx) \to   \mathcal{G}_{\rho}(v)(\bdx)$. This follows from the  convergence of $\rho_n$ to $\rho$  in $L^{1}_{loc}(\mathbb{R}^d)$. We thus conclude that    {for $v\in C^{1}_{c}(\Omega)$,}
 \[
  \mathcal{G}_{\rho_n}(v) \to  \mathcal{G}_{\rho}(v)\quad \text{strongly in $L^{2}$.}
 \]
Thus from the above integration by parts formula we have that for any $v\in C^{1}_{c}(\Omega)$
\[
\begin{aligned}
\lim_{n\to \infty} \int_{\Omega} \mathfrak{D}_{\rho_n}({\bf u}_n) v(\bdx) dx &= -\lim_{n\to \infty}\int_{\Omega} \mathcal{G}_{\rho_n}(v)(\bdx) \cdot {\bf u}_{n}(\bdx) d\bdx \\
&=-\int_{\Omega} \mathcal{G}_\rho(v)(\bdx) \cdot {\bf u}(\bdx) d\bdx\\
&=\int_{\Omega} \mathfrak{D}_{\rho}({\bf u}) v(\bdx) d\bdx.
\end{aligned}
\]
The last inequality is possible because ${\bf u}\in \mathcal{S}_{\rho,2}(\Omega)$. 
Now for any $v\in L^{2}(\Omega)$, let us choose $v_m\in C^{1}_{c}(\Omega)$ such that $v_m\to v$ strongly in $L^{2}(\Omega)$. Then we have for each $n, m$ that 
\[
 \int_{\Omega} \mathfrak{D}_{\rho_n}({\bf u}_n) v(\bdx) d\bdx =  \int_{\Omega} \mathfrak{D}_{\rho_n}({\bf u}_n) v_m(\bdx) dx + R_{n,m}
\]
where 
\[
|R_{n,m}| = \left| \int_{\Omega} \mathfrak{D}_{\rho_n}({\bf u}_n) (v(\bdx)-v_{m}(\bdx)) d\bdx\right|\leq \|\mathfrak{D}_{\rho_n}({\bf u}_n)\|_{L^{2}(\Omega)}\|v_m - v\|_{L^{2}(\Omega)}    {.}
\]
Therefore using the fact that $\|\mathfrak{D}_{\rho_n}({\bf u}_n)\|_{L^{2}(\Omega)}$ is uniformly bounded in $n$, we have that $\lim_{m\to \infty} \sup_{n\in \mathbb{N}} |R_{n,m}|  = 0$ and so we have 
\[
\begin{aligned}
\liminf_{n\to \infty} \int_{\Omega} \mathfrak{D}_{\rho_n}({\bf u}_n) v(\bdx) dx &=\lim_{n\to \infty} \int_{\Omega} \mathfrak{D}_{\rho_n}({\bf u}_n) v_m(\bdx) dx + \liminf_{n\to \infty} R_{n,m}\\
 &= \int_{\Omega} \mathfrak{D}_{\rho}({\bf u}) v_m(\bdx) d\bdx +  \liminf_{n\to \infty} R_{n,m}.  
 \end{aligned}
\]
We now take $m\to \infty$ and use the fact that $\mathfrak{D}_{\rho}({\bf u})\in L^{2}(\Omega)$ to complete the proof the lemma. 
\end{proof}
\begin{lemma}\label{elastic-intermediate}
Suppose that ${\bf u}_n \to {\bf u}$ strongly in $L^{2}(\Omega;\mathbb{R}^d)$, ${\bf u}\in \mathcal{S}_{\rho, 2}(\Omega),$ and that $\sup_{n\in \mathbb{N}}W_{\rho_n} ({\bf u}_n) \leq C<\infty.$ Then it holds that 
 \begin{equation}\label{nlelastiic-part}
\begin{aligned} 
&\int_{\Omega}\int_{\Omega} \rho(\bdx'-\bdx) \left( \mcD({\bf u})(\bdx,\bdx') -\frac{1}{d}\,\mathfrak{D}_{\rho}({\bf u})(\bdx)\right)^2d\bdx'\,d\bdx\\
&\leq  \liminf_{n\to \infty} \int_{\Omega}\int_{\Omega} \rho_n(\bdx'-\bdx) \left( \mcD({\bf u}_n)(\bdx,\bdx') -\frac{1}{d}\,\mathfrak{D}_{\rho_n}({\bf u}_n)(\bdx)\right)^2d\bdx'\,d\bdx. 
\end{aligned}
\end{equation}

\end{lemma}
\begin{proof}
 Let $A\subset \subset \Omega$ and $\varphi\in C^{\infty}_c(B_1(0))$. For $\epsilon < \text{dist}(A, \partial \Omega)$, consider the sequence of functions $\varphi_\epsilon \ast {\bf u}_n$ and $\varphi_\epsilon \ast \mathfrak{D}_{\rho_n}({\bf u}_n)$ defined for $\bdx\in A$,
where 
 $\varphi_\epsilon({\bf z}) = \epsilon^{-d}\varphi({\bf z}/\epsilon)$ is standard mollifiers.  Then since ${\bf u}_n \to {\bf u}$ strongly in $L^{2}$, for a fixed $\epsilon>0, $ we have as $n\to \infty$, 
 \begin{equation}\label{c2-strong-l2}
 \begin{aligned}
 &\varphi_\epsilon \ast {\bf u}_n \to  \varphi_\epsilon \ast  {\bf u}\quad \text{in $C^{2}(\overline{A}; \mathbb{R}^d)$ and}\\
   &\varphi_\epsilon \ast \mathfrak{D}_{\rho_n}({\bf u}_n)\to \varphi_\epsilon \ast \mathfrak{D}_{\rho}({\bf u})\quad \text{strongly in $L^{2}(A)$.}
 \end{aligned}
 \end{equation}
 The    {latter} follows from Lemma \ref{weak-conv-nldiv-seq} and the fact that the convolution is a compact operator.  Using Jensen's inequality, we have that for each $\epsilon >0$ small and $n$ large 
  \begin{equation}\label{Jensen-1st}
 \begin{aligned}
& \int_{A}\int_{A} \rho_n(\bdx'-\bdx) \left( \mcD(\varphi_\epsilon \ast {\bf u}_n)(\bdx,\bdx') -\frac{1}{d}\,\varphi_\epsilon\ast\mathfrak{D}_{\rho_n}({\bf u}_n)(\bdx)\right)^2d\bdx'\,d\bdx\\
 &\leq  \int_{A}\int_{A} \rho_n(\bdx'-\bdx) \left( \mcD({\bf u}_n)(\bdx,\bdx') -\frac{1}{d}\,\mathfrak{D}_{\rho_n}({\bf u}_n)(\bdx)\right)^2d\bdx'\,d\bdx.
 \end{aligned}
 \end{equation}
 The left hand side of \eqref{Jensen-1st} can be rewritten after change of variables as 
 \[
  \begin{aligned}
& \int_{A}\int_{A} \rho_n(\bdx'-\bdx) \left( \mcD(\varphi_\epsilon \ast {\bf u}_n)(\bdx,\bdx') -\frac{1}{d}\,\varphi_\epsilon\ast\mathfrak{D}_{\rho_n}({\bf u}_n)(\bdx)\right)^2d\bdx'\,d\bdx\\
 &=  \int_{\mathbb{R}^{d}} \rho_n(\bdz) \int_{A}\chi_{A}(\bdx+\bdz)\left( \mcD(\varphi_\epsilon \ast{\bf u}_n)(\bdx,\bdx + \bdz) -\frac{1}{d}\,\varphi_\epsilon \ast\mathfrak{D}_{\rho_n}({\bf u}_n)(\bdx)\right)^2\,d\bdx d\bdz.
 \end{aligned}
 \]
Using \eqref{c2-strong-l2}, the sequence of functions
 $$
 \bdz\mapsto \int_{A}\chi_{A}(\bdx+\bdz)\left( \mcD(\varphi_\epsilon \ast{\bf u}_n)(\bdx,\bdx + \bdz) -\frac{1}{d}\,\varphi_\epsilon \ast\mathfrak{D}_{\rho_n}({\bf u}_n)(\bdx)\right)^2d\bdx\
 $$
 converges in $L^{\infty}(\Omega)$    {as $n\to\infty$} to $$\bdz\mapsto \int_{A}\chi_{A}(\bdx+\bdz)\left( \mcD(\varphi_\epsilon \ast{\bf u})(\bdx,\bdx + \bdz) -\frac{1}{d}\,\varphi_\epsilon \ast\mathfrak{D}_{\rho}({\bf u})(\bdx)\right)^2d\bdx$$ where we use the simple inequality $|a^2 - b^2|\leq ||a| + |b|||a-b|$ and the assumption that ${\bf u}\in \mathcal{S}_{\rho,2}(\Omega)$.  Using the convergence of $\rho_n$ to $\rho$ in $L^{1}_{loc}(\mathbb{R}^d)$ and taking the limit in \eqref{Jensen-1st} we conclude that for each $\epsilon > 0$,
 \[\begin{aligned}
& \int_{A}\int_{A} \rho(\bdx'-\bdx) \left( \mcD(\varphi_\epsilon \ast {\bf u})(\bdx,\bdx') -\frac{1}{d}\,\varphi_\epsilon\ast\mathfrak{D}_{\rho}({\bf u})(\bdx)\right)^2d\bdx'\,d\bdx\\
 &\leq \liminf_{n\to \infty} \int_{A}\int_{A} \rho_n(\bdx'-\bdx) \left( \mcD({\bf u}_n)(\bdx,\bdx') -\frac{1}{d}\,\mathfrak{D}_{\rho_n}({\bf u}_n)(\bdx)\right)^2d\bdx'\,d\bdx.
 \end{aligned}
 \]
 Now inequality \eqref{nlelastiic-part} follows after applying first Fatou's lemma in $\epsilon$ and noting that $A\subset\subset \Omega$ was arbitrary.
\end{proof}
Let us state  some elementary inequalities that relate the energy $W_{\rho}({\bf u})$ and its integrand with that of the seminorm $|{\bf u}|_{\mathcal{S}_{\rho, 2}}$. The proof follows from direct calculations and uses a simple application of H\"older's inequality.  
\begin{lemma}\label{W-integrand}For a given $\rho\in L^{1}_{loc}(\mathbb{R}^d)$ and $\Omega$ bounded such that for ${\bf u}\in \mathcal{S}_{\rho,2} (\Omega)$ and $\bdx \in \Omega$ we have 
\[
\begin{aligned}
\mathfrak{D}_{\rho}({\bf u})^{2}(\bdx) &\leq \|\rho\|_{L^{1}(B_{R}({\bfs 0}))}\int_{\Omega}\rho(\bdx-\bdy)|\mcD({\bf u})(\bdx, \bdy)|^2 d\bdy\\
\int_{\Omega} \rho(\bdy-\bdx) \left( \mcD({\bf u})(\bdx,\bdy) -\frac{1}{d}\,\mathfrak{D}_{\rho}({\bf u})(\bdx)\right)^2d\bdy& \leq C(d,\|\rho\|_{L^{1}(B_{R}({\bfs 0}))} )\,\int_{\Omega}\rho(\bdx-\bdy)|\mcD({\bf u})(\bdx, \bdy)|^2 d\bdy.
\end{aligned}
\]
Moreover, we have positive constants $C_1$ and $C_2$, depending on $\rho, \Omega,$ and $d$, such that for all ${\bf u}\in \mathcal{S}_{\rho,2} (\Omega)$ 
\[
C_{1} |{\bf u}|_{\mathcal{S}_{\rho,2} (\Omega)}^{2} \leq W_{\rho}({\bf u}) \leq C_2 |{\bf u}|_{\mathcal{S}_{\rho,2} (\Omega)}^{2}.
\]
\end{lemma}

 \begin{proof}[Proof Theorem \ref{thm-G-convergence-intro}]
 The proof has two parts: the demonstration of the $\Gamma$-convergence of the energy functionals and the proof of the convergence of minimizers. For the first part, following the definition of $\Gamma$-convergence, we prove the two items in Definition \ref{defn-gamma}.
 
 {\bf Item a)}. Suppose that $   {{\bf u}_n} \to {\bf u}$ strongly in $L^{2}$.  We will show that  \[\overline{E}_{\rho}({\bf u})\leq \liminf_{n\to \infty}\overline{E}_{n}({\bf u}_{n})    {.}\]
 Since $\int_{\Omega}{\bf f}\cdot {\bf u}_nd\bdx \to \int_{\Omega}{\bf f}\cdot {\bf u}d\bdx$ as $n\to \infty,$ we only need to show that 
 \[
 {W}_{\rho}({\bf u})\leq \liminf_{n\to \infty}{W}_{\rho_n}({\bf u}_{n})    {.}
 \]
 To that end, we will assume without loss of generality that $ \liminf_{n\to \infty}{W}_{\rho_{n}}({\bf u}_{n})<\infty$, and so (up to a subsequence) $\sup_{n\in \mathbb{N}}W_{\rho_n} ({\bf u}_n) \leq C<\infty.$ Then we have that $\{\mathfrak{D}_{\rho_n}({\bf u}_n)\}$ is uniformly bounded in $L^{2}(\Omega)$ and $\{|{\bf u}_{n}|_{\mathcal{S}_{\rho_n, 2}}\}$ is uniformly bounded as well, by Lemma \ref{W-integrand}. To prove the desired inequality it suffices to show that 
 \begin{equation}\label{nldiv-seq}
 \int_{\Omega} (\mathfrak{D}_{\rho}({\bf u}))^{2} dx\leq \liminf_{n\to \infty}   \int_{\Omega} (\mathfrak{D}_{\rho_n}({\bf u}_n))^{2} dx
\end{equation}
and 
\begin{equation}\label{nlelastiic-part-n}
\begin{aligned} 
&\int_{\Omega}\int_{\Omega} \rho(\bdx'-\bdx) \left( \mcD({\bf u})(\bdx,\bdx') -\frac{1}{d}\,\mathfrak{D}_{\rho}({\bf u})(\bdx)\right)^2d\bdx'\,d\bdx\\
&\leq 
\liminf_{n\to \infty} \int_{\Omega}\int_{\Omega} \rho_n(\bdx'-\bdx) \left( \mcD({\bf u}_n)(\bdx,\bdx') -\frac{1}{d}\,\mathfrak{D}_{\rho_n}({\bf u}_n)(\bdx)\right)^2d\bdx'\,d\bdx.
\end{aligned}
\end{equation}
To show \eqref{nldiv-seq}, using the weak lower semicontinuity of norm, 
 it suffices to show that $\mathfrak{D}_{\rho_n}({\bf u}_n) \rightharpoonup \mathfrak{D}_{\rho}({\bf u})$ weakly in $L^{2}(\Omega)$. But this is proved in Lemma \ref{weak-conv-nldiv-seq} after noting    {the} above assumption.  
  
  Inequality \eqref{nlelastiic-part-n} will follow from Lemma \ref{elastic-intermediate} if we show ${\bf u} \in \mathcal{S}_{\rho,2}(\Omega)$.  But under the assumptions on the sequence ${\bf u}_n$, the conclusion ${\bf u} \in \mathcal{S}_{\rho,2}(\Omega)$ follows from Theorem \ref{compactness-sequence} that will be proved in the coming sections    {.}

{\bf Item b)}. For a given ${\bf u}\in L^{2}(\Omega)$, we take the recovery sequence to be ${\bf u}_{n} = {\bf u}$. Now if ${\bf u}\in L^{2}(\Omega)\setminus \mathcal{S}_{\rho,2}(\Omega)$, then by definition $\bar{E}_{\infty}({\bf u}) = \infty$ and necessarily $ \liminf_{n\to \infty} E_{\rho_{n}}({\bf u}) =\infty$. Otherwise, up to a subsequence (nor renamed) $\sup_{n} E_{\rho_n}({\bf u})  < \infty$ and 
 \[
 [{\bf u}]_{\mathcal{S}_{\rho_n, 2}}^2 \leq C (E_{\rho_{n}}({\bf u}) + \|{\bf u}\|_{L^{2}}) \leq C +  \|{\bf u}\|_{L^{2}},
 \]
 where we used Lemma \ref{W-integrand}. 
    {Then by passing to the limit and} using Fatou's lemma, we have $ [{\bf u}]_{\mathcal{S}_{\rho, 2}}^2 < \infty$, that is, ${\bf u}\in \mathcal{S}_{\rho,2}(\Omega)$, which is a contradiction. 
In the event     {that} ${\bf u}\in \mathcal{S}_{\rho, 2}(\Omega)$, we may use \eqref{condition-on-seq-ker-p=2} to get the pointwise convergence and Lemma \ref{W-integrand} to get appropriate bounds of the integrand of $W_{\rho_n}({\bf u})$ to apply the Dominated Convergence Theorem and conclude that  $ \liminf_{n\to \infty} W_{\rho_{n}}({\bf u}) =W_{\rho}({\bf u})$, from which Item b) follows. 

We next prove the second part of the theorem, the convergence of minimizers. To apply \cite[Theorem 7.8 and Corollary 7.202]{DalMaso}, we need to prove     {the} equicoercvity of the functionals restricted to $V\cap \mathcal{S}_{\rho_n, 2}$. That is, for ${\bf u}_n \in V\cap \mathcal{S}_{\rho_n, 2}$ such that $\sup_{n \geq 1} E_{\rho_n}({\bf u}_n) < \infty$, we need to show that the sequence $\{{\bf u}_{n}\}$ is precompact in $L^{2}(\Omega;\mathbb{R}^d)$. To that end, first a positive constant $C>0$ and for all $n\geq 1$
\[
 [{\bf u}_n]_{\mathcal{S}_{\rho_n, 2}}^2 \leq C (E_{\rho_{n}}({\bf u}_n) + \|{\bf u}_n\|_{L^{2}}) \leq C +  \|{\bf u}_n\|_{L^{2}}
\]
Using the uniform Poincar\'e-Korn inequality, Theorem \ref{Poincare}, for all large $n$ we have that $ \|{\bf u}_n\|_{L^{2}} \leq C  [{\bf u}_n]_{\mathcal{S}_{\rho_n, 2}}$ and as a consequence 
\[
 [{\bf u}_n]_{\mathcal{S}_{\rho_n, 2}}^2 \leq C(1 +  [{\bf u}_n]_{\mathcal{S}_{\rho_n, 2}})\quad \quad\text{for all $n$ large}.
\]
It then follows that $ [{\bf u}_n]_{\mathcal{S}_{\rho_n, 2}}$ is uniformly bounded and therefore, by the uniform Poincar\'e-Korn inequality, $\|{\bf u}_n\|_{L^{2}(\Omega)}$ is bounded as well.  We now use the compactness result, Theorem    {\ref{compactness-sequence}}, to conclude that $\{{\bf u}_n\}$ is precompact in $L^{2}(\Omega;\mathbb{R}^2)$ with limit point ${\bf u}$ in $\mathcal{S}_{\rho, 2}(\Omega) \cap V.$   We may now apply \cite[Theorem 7.8 and Corollary 7.202]{DalMaso} to state that ${\bf u}$ is a minimizer of $E_{\rho}$ over $\mathcal{S}_{\rho, 2}(\Omega) \cap V.$ 
 \end{proof}

\section{Compactness in $L^{p}_{loc}(\mathbb{R}^{d})$}
\label{sec2}
The proof of the $L^{p}$ compactness result, Theorem \ref{main-compactness}, will be carried out in two steps. We establish first compactness in $L^{p}_{loc}$ topology followed by proving     {a} boundary estimate that controls growth near the boundary of the domain. 
The $L^{p}_{loc}$ compactness will be proved in this section under a weaker assumption on the kernel. 
In fact $L^{p}_{loc}$ compactness will be stated and proved for a broader class of kernels that include kernels of the type $\tilde{\rho} (   {\bfs\xi})\chi_{B_{1}^{\Lambda}}(   {\bfs\xi})$ where $\tilde{\rho}$ satisfies  \eqref{rho-inc-rad} and \eqref{suff-cond}, where $B_{1}^{\Lambda} = \{\bdx\in B_{1}: {\bdx/|{\bdx}|}\in \Lambda\}$ is a conic region spanned by 
 a given a nontrivial spherical cap $\Lambda \subset\mathbb{S}^{d-1}$.
 To make this and the condition of the theorem  precise, we begin identifying the kernel $\rho$ by the representative 
\[
{\rho}(\bdx) = \left\{\begin{aligned}
&\lim_{h\to 0}\fint_{B_{h}(\bdx)} \rho({\bfs \xi})d{{\bfs \xi}},\quad \text{if $\bdx$ is a Lebesgue point},\\
& \infty,\quad \text{otherwise}.
\end{aligned}\right.
\]
For $\theta_{0}\in (0,1)$ and ${\bf v}\in \mathbb{S}^{d-1}$, let us define 
\[
\rho_{\theta_{0}}(r{\bf v}) = \inf_{\theta\in [\theta_{0}, 1]}\rho(\theta r{\bf v})\theta^{-p}. 
\]
It is clear that for a given ${\bf v}\in \mathbb{S}^{d-1}$, $\rho_{\theta_{0}}(r{\bf v}) \leq \rho(\theta r{\bf v})\theta^{-p}$ for any $\theta\in [\theta_{0}, 1]$ and $r\in (0, \infty)$.  In particular, this implies $\rho_{\theta_{0}}({\bfs \xi}
) \leq \rho({\bfs \xi})$ for any ${\bfs \xi}$, with the equality holds if $\rho$ is radial  and $|{\bfs \xi}|^{-p}\rho({\bfs \xi})$ is {\em nonincreasing} in $|{\bfs \xi}|$.
 
We now make a main assumption on $\rho$ that  
\begin{equation}\label{cone-suff-condition}
{\small \begin{aligned}
\quad\quad &\text{$\exists\, \theta_{0}\in (0, 1)$,  $\Lambda \subset \mathbb{S}^{d-1}$ and ${\bf v}_0\in \Lambda$
such that 
$\mathcal{H}^{d-1}(\Lambda) > 0$,}\\ 
& \text{$\rho_{\theta_{0}}(r{\bf v})  =\rho_{\theta_{0}}(r{\bf v}_0)$, $\forall\, (r, {\bf v}) \in (0, \infty)\times\Lambda$,\,  and  }\;
\lim_{\delta \to 0}\frac{\delta^{p}}{\displaystyle \int_{0}^{\delta} \rho_{\theta_{0}}(r{\bf v}_0) r^{d-1} dr}  = 0.
\end{aligned}}
\end{equation}

   {Assumption} \eqref{cone-suff-condition} says that, on a conic region with apex at the origin, the kernel $\rho$ is above a nonnegative function 
with appropriate singular growth near the origin.  Note that on one hand, it is not difficult to see if $\rho \in L^{1}_{loc}(\mathbb{R}^{d})$ is a  nonnegative function that satisfies \eqref{rho-inc-rad} and \eqref{suff-cond}, then it also satisfies \eqref{cone-suff-condition}. On the other hand, if  $\tilde{\rho}$ satisfies  \eqref{rho-inc-rad} and \eqref{suff-cond}, then given a nontrivial spherical cap $\Lambda$ and  conic region $B_{1}^{\Lambda} = \{\bdx\in B_{1}: {\bdx/|{\bdx}|}\in \Lambda\}$, the kernel $\rho({\bfs \xi}) = \tilde{\rho} ({\bfs \xi})\chi_{B_{1}^{\Lambda}}({\bfs \xi})$  satisfies \eqref{cone-suff-condition} (with $\theta_0$ being any number in $(0, 1)$ and ${\bf v}_0$ representing any vector in $\Lambda$) but not necessarily \eqref{rho-inc-rad} and \eqref{suff-cond}.  For kernels of this form, we need the formulation in \eqref{cone-suff-condition} to carry out the proof of  $L^{p}_{loc}$ compactness. We should also note that one can construct other $\rho$ that are not of the above form that satisfy \eqref{cone-suff-condition}, see \cite[equation (17)]{anotherlook}.

 \begin{theorem}[$L^{p}_{loc}$ compactness]\label{loc-compactness-main}Suppose that $1\leq p  < \infty$. 
Let $\rho \in L^{1}(\mathbb{R}^{d})$ be a nonnegative function satisfying \eqref{cone-suff-condition}. Suppose also that $\{{\bf u}_{n}\}$ is a sequence of vector fields that is bounded in $\mathcal{S}_{\rho,p}(\mathbb{R}^{d})$. 
Then for any $D\subset \mathbb{R}^{d}$ open and bounded,  the sequence $\{{\bf u}_{n}|_{D}\}$ is precompact in $L^{p}(D; \mathbb{R}^{d})$. 
\end{theorem}
We should mention that although the focus is different, operators that use non-symmetric kernels like those satisfying the condition \eqref{cone-suff-condition} have been studied in connection with semi-Dirichlet forms and the processes they generate, see \cite{Kassmann,Barles} for more discussions.  In particular, most of the examples of kernels listed in \cite[Section 6]{Kassmann} satisfy condition \eqref{cone-suff-condition}.

\subsection{A few technical lemmas}
We begin with the following lemma whose proof can be carried out following the argument used in \cite{Ponce2003}. 
Let ${\bf u}\in L^{p}(\mathbb{R}^{d};\mathbb{R}^{d})$ be given, we introduce the function  $F_{p}[{\bf u}]:\mathbb{R}^{d}\to [0, \infty)$ defined by 
\[
F_{p}[{\bf u}]({\bf h}) = \int_{\mathbb{R}^{d}}\left|({\bf u}(\bdx + {\bf h}) - {\bf u}(\bdx))\cdot \frac{{\bf h}}{|{\bf h}|}\right|^{p}d{\bdx},\quad \text{for ${\bf h}\in \mathbb{R}^{d}$}. 
\]
\begin{lemma}\label{est-for-F} Suppose that $\theta_{0}$ is given as  in \eqref{cone-suff-condition}. 
There exists a constant $C = C(\theta_{0},p)>0$ such that for any $\delta > 0$, and  ${\bf v}\in \mathbb{S}^{d-1}$ 
\[
\begin{aligned}
F_{p}[{\bf u}](t{\bf v}) 
&\leq C \frac{\delta^{p}}{\displaystyle\int_{0}^{\delta} \rho_{\theta_{0}} (s{\bf v})s^{d-1} ds}  \int_{0}^{\infty}\rho(h{\bf v}) h^{d-1}\frac{F_{p}[{\bf u}](h{\bf v})}{{h^{p}}}dh,
\end{aligned}
\]
for any $0 < t < \delta$ and any  ${\bf u}\in L^{p}(\mathbb{R}^{d}, \mathbb{R}^{d})$. 
\end{lemma}
\begin{proof}
For any ${\bf v}\in \mathbb{S}^{d-1}$ and $t\in \mathbb{R}$, we may rewrite the function $F_{p}$ as 
\[
F_{p}[{\bf u}](t{\bf v})= \int_{\mathbb{R}^{d}}|({\bf u}(\bdx + t{\bf v}) - {\bf u}(\bdx))\cdot {\bf v}|^{p}d{\bdx}. 
\]
It follows from \cite[Lemma 3.1]{Ponce2003} that given $0 <s<t$, there exist $C_{p}$ and $\theta = \frac{t}{s} - k\in (0,1)$ ($k$ an integer) such that 
\[
\frac{F_{p}[{\bf u}](t{\bf v})}{ t^{p}} \leq C_{p}\left\{ {F_{p}[{\bf u}](s{\bf v})\over s^{p}}  + {F_{p}[{\bf u}](\theta s{\bf v})\over t^{p}}\right\}.
\]
 We also have that for a given $l_{0} \in \mathbb{N}$, 
 \[
 F_{p}[{\bf u}](\theta s{\bf v}) \leq l_{0}^{p} F_{p}[{\bf u}]\left(\frac{\theta s}{ l_{0}} {\bf v}\right) \leq {2^{(p-1)}}l^{p}_{0}\left\{ F_{p}[{\bf u}](s{\bf v})  + F_{p}[{\bf u}]\left(s - \frac{s\theta}{{l_{0}}}{\bf v}\right)\right\}. 
 \]
 Combining the above we have that for any $l_{0}$, there exists a constant $C = C(p, l_{0})$ such that 
 \begin{equation}\label{lp-loc-eqn1}
\frac{F_{p}[{\bf u}](t{\bf v})}{ t^{p}} \leq C(p, l_{0})\left\{ {F_{p}[{\bf u}](s{\bf v})\over s^{p}}  + {F_{p}[{\bf u}](\tilde{\theta} s{\bf v})\over t^{p}}\right\},\quad \text{where $\tilde{\theta} = 1- {\theta\over l_{0}}$. }
\end{equation}
Now let us take $\theta_{0}$ as given in \eqref{cone-suff-condition} and choose $l_{0}$ large that ${1\over l_{0}} < 1-\theta_{0}$.  It  follows that $\theta_{0} < \tilde{\theta} \leq 1$. 
Then for any $\delta > 0$, and any $0 < s  < \delta \leq \tau $, by multiplying both sides of inequality \eqref{lp-loc-eqn1} by $\rho_{\theta_{0}} ({\bf v}s)$ and integrating from $0$ to $\delta$, we obtain 
\[
\begin{aligned}
&
\int_{0}^{\delta} \rho_{\theta_{0}}(s{\bf v}) s^{d-1}ds  {F_{p}[{\bf u}](\tau{\bf v})\over \tau^{p}} \leq C(p, l_{0})\left\{ \int_{0}^{\delta} \rho_{\theta_{0}}(s{\bf v}) s^{d-1}{F_{p}[{\bf u}](s{\bf v})\over s^{p}} ds  \right.\\
&\qquad\quad \left. +  \int_{0}^{\delta} \rho_{\theta_{0}}(s{\bf v}) s^{d-1}{F_{p}[{\bf u}](\tilde{\theta} s{\bf v})\over \tau^{p}}ds \right\}.
\end{aligned}
\]
Let us estimate the second integral in the above: 
\[
I =  {1\over \tau^{p}}\int_{0}^{\delta} \rho_{\theta_{0}}(s{\bf v}) s^{d-1}{F_{p}[{\bf u}](\tilde{\theta} s{\bf v})}ds. 
\]
We first note that using the definition of $\rho_{\theta_{0}}$ and since $\delta \leq \tau$, we have 
\[
I\leq  {1 \over {\theta_{0}^{d-1}}} \int_{0}^{\tau} \rho(\tilde{\theta}s{\bf v}) (\tilde{\theta} s)^{d-1}{F_{p}[{\bf u}](\tilde{\theta} s{\bf v})\over {(\tilde{\theta} s)^{p}}}ds. 
\]
The intention is to change variables $h=\tilde{\theta} s$. However, note that $\tilde{\theta}$ is a function of $s$, and by definition 
\[
\tilde{\theta}s = \left(\frac{k}{ l_{0}} + 1\right)s - \frac{\tau}{ l_{0}}\quad \text{for $k \leq \frac{\tau}{s}  < k + 1$}.
\]
It then follows by a change of variables that
\[
\begin{aligned}
I &\leq  \frac{1} {{\theta_{0}^{d-1}}} \sum_{k=1}^{\infty}\int_{\frac{\tau}{ (k+1)}}^{\frac{\tau}{ k}}  \rho(\tilde{\theta}s{\bf v}) (\tilde{\theta}s)^{d-1}\frac{F_{p}[{\bf u}](\tilde{\theta} s{\bf v})}{ {(\tilde{\theta} s)^{p}}}ds\\
& = \frac{1}{{\theta_{0}^{d-1}}} \sum_{k=1}^{\infty}\int_{{\tau(1 - \frac{1}{ l_{0}})\over (k+1)}}^{\frac{\tau}{ k}}  \rho(h{\bf v}) h^{d-1}{F_{p}[{\bf u}](h{\bf v})\over {h^{p}}}{dh\over {k\over l_{0}} + 1}\\
&\leq C \int_{0}^{\infty}\rho(h{\bf v}) h^{d-1}\frac{F_{p}[{\bf u}](h{\bf v})}{h^{p}}dh,
\end{aligned}
\]
where in the last estimate integrals over overlapping domains were counted at most a finite number of times. 
Combining the above estimates we have shown that there exists a constant $C$ such that for any ${\bf v} \in \mathbb{S}^{d-1}$,  $\delta>0$ and $\tau\geq \delta$
\[
\left(\int_{0}^{\delta} \rho_{\theta_{0}} (s{\bf v})s^{d-1} ds\right) \frac{F_{p}[{\bf u}](\tau{\bf v})}{ \tau^{p}} \leq C \int_{0}^{\infty}\rho(h{\bf v}) h^{d-1}\frac{F_{p}[{\bf u}](h{\bf v})}{ {h^{p}}}dh. 
\]
Rewriting the above and restricting ${\bf v} \in \Lambda$ we have that 
\[
F_{p}[{\bf u}](\tau{\bf v}) \leq C \frac{\tau^{p}}{ \displaystyle \int_{0}^{\delta} \rho_{\theta_{0}} (s{\bf v})s^{d-1} ds}  \int_{0}^{\infty}\rho(h{\bf v}) h^{d-1}\frac{F_{p}[{\bf u}](h{\bf v})}{h^{p}}dh. 
\]
Now let $0 < t < \delta$ and applying the above inequality for $\tau = \delta$ and $\tau = t + \delta$, we obtain
\[
\begin{aligned}
F_{p}[{\bf u}](t{\bf v}) &=  F_{p}[{\bf u}]((t + \delta){\bf v} -\delta{\bf v})\\
&\leq 2^{p-1} \left\{F_{p}[{\bf u}]((t + \delta){\bf v}) + F_{p}[{\bf u}](\delta{\bf v})\right\}\\
&\leq C \frac{\delta^{p}}{ \displaystyle   \int_{0}^{\delta} \rho_{\theta_{0}} (s{\bf v})s^{d-1} ds}  \int_{0}^{\infty}\rho(h{\bf v}) h^{d-1}\frac{F_{p}[{\bf u}](h{\bf v})}{ {h^{p}}}dh. 
\end{aligned}
\]
This completes the proof. 
\end{proof}

\begin{lemma}
    \label{est-for-F-cor}
Suppose that $\rho\in L^{1}_{loc}(\mathbb{R}^{d})$ and there exists a spherical cap $\Lambda\subset \mathbb{S}^{d-1}$ and a vector ${\bf v}_{0}\in\Lambda$
 such that the function $\rho(r{\bf v}) =\rho(r{\bf v}_{0}) =
  \tilde{\rho}(r)$, for all ${\bf v}\in \Lambda$ and  $r\mapsto r^{-p}\tilde{\rho}(r)$ is nonincreasing.  Then there exists a constant $C = C(d,p,\Lambda)$ such that for any $\delta > 0$, and  ${\bf v}\in \Lambda$, 
\[
\begin{aligned}
F_{p}[{\bf u}](t{\bf v}) 
&\leq C \frac{\delta^{p}}{\displaystyle  \int_{0}^{\delta} \tilde{\rho} (s)s^{d-1} ds}  \int_{0}^{\infty}{\rho}(h{\bf v}) h^{d-1}\frac{F_{p}[{\bf u}](h{\bf v})}{{h^{p}}}dh,  
\end{aligned}
\]
for any $0 < t < \delta$ and any  ${\bf u}\in L^{p}(\mathbb{R}^{d}, \mathbb{R}^{d})$. 
\end{lemma}
\begin{proof}
It suffices to note that for $\rho\in L^{1}_{loc}(\mathbb{R}^{d})$ that satisfies the conditions in the statement of the proposition, we have that for any $\theta_{0} \in (0, 1)$, and any ${\bf v}\in \Lambda$,
\[
 \rho_{\theta_{0}} (r{\bf v}) = r^{p} \inf_{\theta\in [\theta_{0},1]} \rho(\theta r{\bf v})(\theta r)^{-p} = \rho(r{\bf v})  =\rho(r{\bf v}_{0}) =\tilde{\rho}(r).
\]
We may then repeat the argument in the proof of Lemma \ref{est-for-F}. 
\end{proof}

Before  proving one  of the main results,
 we make an elementary observation. \begin{lemma}\label{positive-inf-on-sector}
Let $1\leq p<\infty.$ Given a spherical cap $\Lambda$ with aperture $\theta$, there exists a positive constant $c_{0},$ depending only on $d, \theta$ and $p$, such that
\begin{equation*}
\inf_{{\bf w}\in \mathbb{S}^{d-1}}\int_{\Lambda\cap \mathbb{S}^{d-1}} |{\bf w}\cdot {\bf s}|^{p}d\sigma({\bf s}) \geq c_{0} > 0.
\end{equation*}
\end{lemma}
The above lemma follows from the fact that the map
$$
{\bf w}\mapsto \int_{\Lambda\cap\mathbb{S}^{d-1}} |{\bf w}\cdot {\bf s}|^{p}d\sigma({\bf s})
$$
is continuous on the compact set $\mathbb{S}^{d-1}, $ and is positive, for otherwise the portion of the unit sphere $\Lambda $ will be orthogonal to a fixed vector which is not possible since $\mathcal{H}^{d-1}(\Lambda) > 0$. 

\subsection{Proof of Theorem  \ref{loc-compactness-main}}
From the assumption we have 
\begin{equation}\label{mainconditioncompact}
\sup_{n\geq 1} \|{\bf u}_{n}\|_{L^{p}}^{p} + \sup_{n\geq 1}\int_{\mathbb{R}^{d}}\int_{\mathbb{R}^{d}}\rho(\bdx'-\bdx)\left|\mcD({\bf u}_n)(\bdx,\bdx')\right|^{p}d\bdx'd\bdx < \infty.
\end{equation}
We will use the compactness criterion in \cite[Lemma 5.4]{Mengesha}, which is a variant of the well-known Riesz-Fr\'echet-Kolmogorov compactness criterion \cite[Chapter IV.27]{brezis-Functional}.
Let $\Lambda$ be as given in \eqref{cone-suff-condition}.  For $\delta>0$, let us introduce the matrix $Q = (q_{ij})$, where 
\[
q_{ij} = \int_{\Lambda} s_{i}s_{j} d \mathcal{H}^{d-1}({\bf s})
\]
The symmetric matrix $\mathbb{Q}$ is invertible. Indeed, the smallest eigenvalue is given by 
$$ \lambda_{min} = \min_{|\bdx| = 1} \langle \mathbb{Q}\bdx, \bdx \rangle = \min_{|\bdx| = 1} \int_{\Lambda} |\bdx\cdot {\bf s}|^{2} d \mathcal{H}^{d-1}({\bf s}) $$
 which we know is positive by Lemma \ref{positive-inf-on-sector}. 
We  define the following matrix functions   
\begin{equation*}
\mathbb{P} ({\bf z}) = d\mathbb{Q}^{-1}\frac{{\bf z}\otimes {\bf z}}{|{\bf z}|^{2}}\chi_{B^{\Lambda}_{1}}({\bf z}), \quad \quad \mathbb{P}^{\delta}({\bf z}) = \delta^{-d}\mathbb{P} \left(\frac{{\bf z}}{\delta}\right)
\end{equation*}
where $B_{1}^{\Lambda} = \{\bdx\in B_{1}: {\bdx/|{\bdx}|}\in \Lambda\}$, as defined before.  
Then for any $\delta>0$, 
\[
\int_{\mathbb{R}^{d}} \mathbb{P}^{\delta}({\bf z}) d{\bf z} = \mathbb{I}.
\]
To prove the theorem, using  \cite[Lemma 5.4]{Mengesha}, it suffices to prove that 
\begin{equation}\label{compact-closeness}
\lim_{\delta\to 0}\limsup_{n\to \infty}\|{\bf u}_{n} - \mathbb{P}^{\delta}*{\bf u}_{n}\|_{L^{p}(\mathbb{R}^{d})} = 0. 
\end{equation}
We show next that  the inequality \eqref{mainconditioncompact}  and condition \eqref{cone-suff-condition} imply \eqref{compact-closeness}.  
To see this, we begin by introducing the notation $B_{\delta}^{\Lambda} = \{\bdx\in B_{\delta}({\bfs 0}): {\bdx/|{\bdx}|}\in \Lambda\}$  and 
applying Jensen's inequality to get
\begin{equation}\label{boundwithF}
\begin{aligned}
\int_{\mathbb{R}^{d}}& |{\bf u}_{n}(\bdx) - \mathbb{P}^{\delta}*{\bf u}_{n}(\bdx)|^{p}d\bdx \leq \int_{\mathbb{R}^{d}}\left|\int_{\mathbb{R}^{d}}\mathbb{P}^{\delta}(\bdy -\bdx)({\bf u}_{n}(\bdy) - {\bf u}_{n}(\bdx)) d\bdy\right|^{p}d\bdx\\
&\leq \int_{\mathbb{R}^{d}}\left||\Lambda|\mathbb{Q}^{-1}\fint_{B^{\Lambda}_{\delta}(\bdx)}\frac{(\bdy-\bdx)}{|\bdy-\bdx|}\cdot({\bf u}_{n}(\bdy) - {\bf u}_{n}(\bdx)) \frac{(\bdy-\bdx)}{|\bdy-\bdx|}d\bdy\right|^{p}d\bdx\\
&\leq |\Lambda|^{p}\|\mathbb{Q}^{-1}\|^{p}\int_{\mathbb{R}^{d}}\left|\fint_{B^{\Lambda}_{\delta}(\bdx)}\frac{(\bdy-\bdx)}{|\bdy-\bdx|}\cdot({\bf u}_{n}(\bdy) - {\bf u}_{n}(\bdx)) \frac{(\bdy-\bdx)}{|\bdy-\bdx|}d\bdy\right|^{p}d\bdx\\
&\leq\frac{ |\Lambda|^{p}\|\mathbb{Q}^{-1}\|^{p}}{|B^{\Lambda}_{\delta}|}\int_{0}^{\delta}\int_{\Lambda}\tau^{d-1}{F}_{p}[{\bf u}_{n}](\tau{\bf v}) d \mathcal{H}^{d-1}({\bf v})d\tau\\
&\leq \frac{C(d,p)}{|B^{\Lambda}_{\delta}|} \int_{0}^{\delta}\int_{\Lambda}\tau^{d-1}{F}_{p}[{\bf u}_{n}](\tau{\bf v}) d \mathcal{H}^{d-1}({\bf v})d\tau
\end{aligned}
\end{equation} 
where as defined previously
\[
{F}_{p}[{\bf u}_{n}](\tau {\bf v})=\int_{\mathbb{R}^{d}} \left|{\bf v}\cdot({\bf u}_{n}(\bdx+\tau {\bf {\bf v}}) - {\bf u}_{n}(\bdx))\right|^{p}d\bdx.\]
Moreover, 
 the fact that $|\Lambda|^{p}\|\mathbb{Q}^{-1}\|^{p} \leq C(d,p, \Lambda)$ for any $\delta > 0$ is also used. 
 We can now apply Lemma \ref{est-for-F} and use the condition \eqref{cone-suff-condition} to obtain that 
\[
\begin{aligned}
& \frac{ C(d,p,\lambda)}{|B^{\Lambda}_{\delta}|}\int_{0}^{\delta}\int_{\Lambda}\tau^{d-1}{F}_{p}[{\bf u}_{n}](\tau{\bf v}) d \mathcal{H}^{d-1}({\bf v})d\tau\\
&\quad \leq  \frac{ C(d,p,\Lambda)}{|B^{\Lambda}_{\delta}|}\int_{0}^{\delta}\tau^{d-1} d \tau\int_{\Lambda} 
\left(
{\displaystyle
 \frac{\delta^{p}}{\displaystyle \int_{0}^{\delta} \rho_{\theta_{0}} (s{\bf v}_{0})s^{d-1} ds } \int_{0}^{\infty}\rho(h{\bf v}) h^{d-1}\frac{F_{p}[{\bf u}_{n}](h{\bf v})}{{h^{p}}}dh} 
\right) d\mathcal{H}^{d-1}({\bf v})\\
&\quad \leq C(d, p,\Lambda) \frac{\delta^{p}}{ \displaystyle \int_{0}^{\delta} \rho_{\theta_{0}} (s{\bf v}_{0})s^{d-1} ds}  |{\bf u}_{n}|_{\mathcal{S}_{\rho,p}(\mathbb{R}^{d})}. 
\end{aligned}
\]
Therefore from the boundedness assumption \eqref{mainconditioncompact} we have, 
 \[
 \begin{aligned}
  \int_{\mathbb{R}^{d}}&|{\bf u}_{n}(\bdx) - \mathbb{P}^{\delta}*{\bf u}_{n}(\bdx)|^{p}d\bdx  
  \leq C(p,d,\Lambda) \frac{\delta^{p}}{ \displaystyle
 \int_{0}^{\delta} \rho_{\theta_{0}} (s{\bf v}_{0})s^{d-1} ds}. 
  \end{aligned}\]
Equation \eqref{compact-closeness} now follows from condition \eqref{cone-suff-condition} after letting $\delta\to 0$. 
That completes the proof.

\subsection{A variant of compactness in $L^{p}_{loc}(\mathbb{R}^{d};\mathbb{R}^{d})$ }

A corollary of the compactness result, Theorem  \ref{loc-compactness-main}, 
is the following result that uses a criterion involving a sequence of kernels.
The effort made in the proof above was to show the theorem for kernel $\rho$ satisfying 
 \eqref{cone-suff-condition}, but the proposition below limits to those satisfying
  \eqref{rho-inc-rad} and  \eqref{suff-cond}.

 \begin{proposition}\label{compactness-sequence-loc}
   Let $\rho \in L^{1}_{loc}$ satisfy \eqref{rho-inc-rad} and  \eqref{suff-cond}.
Let  $\rho_{n}$ be a sequence of radial functions  satisfying \eqref{rho-inc-rad} and that 
 $ \rho_{n}\rightharpoonup \rho$ weakly in $L^{1}$ as $n\to \infty$. 
 If  
 \[
 \sup_{n\geq 1}\{ \|{\bf u_{n}}\|_{L^{p}(\mathbb{R}^{d})}  + |{\bf u}_{n}|_{\mathcal{S}_{\rho_{n}, p}}\} < \infty
 \]
 then  $\{{\bf u}_{n}\}$ is precompact in $L^{p}_{loc}(\mathbb{R}^{d};\mathbb{R}^{d})$. Moreover, if $A\subset \mathbb{R}^{d}$ is a  compact subset, the limit point of the sequence     {restricting} to $A$ is in $\mathcal{S}_{\rho,p}(A).$ 
 \end{proposition}
\begin{proof}
Using Lemma \ref{est-for-F-cor} applied to each $\rho_{n}$, we can repeat the argument in the proof of Theorem \ref{loc-compactness-main}  to obtain 
\[
\int_{\mathbb{R}^{d}}|{\bf u}_{n}(\bdx) - \mathbb{P}^{\delta}*{\bf u}_{n}(\bdx)|^{p}d\bdx  
  \leq C(p, d) \frac{\delta^{p}}{ \displaystyle
 \int_{0}^{\delta} \rho_{n}(r)r^{d-1} dr  }  \leq   C(p, d) \frac{\delta^{p}}{ \displaystyle
\int_{B_{\delta}} \rho_{n}({\bfs \xi}) d{\bfs \xi}  } \,.
\]
Now since $ \rho_{n}\rightharpoonup \rho$, weakly in $L^{1}$ as $n\to \infty$, for a fixed $\delta>0$, it follows that 
\[
\limsup_{n\to \infty} \int_{\mathbb{R}^{d}}|{\bf u}_{n}(\bdx) - \mathbb{P}^{\delta}*{\bf u}_{n}(\bdx)|^{p}d\bdx  
  \leq C(p,d) \frac{\delta^{p}}{ \displaystyle
 \int_{B_{\delta}} \rho({\bfs \xi}) d{\bfs \xi} } \,.
\]
We now let $\delta \to 0$, and use the assumption \eqref{suff-cond} to obtain 
\[
\lim_{\delta\to 0}\limsup_{n\to \infty} \int_{\mathbb{R}^{d}}|{\bf u}_{n}(\bdx) - \mathbb{P}^{\delta}*{\bf u}_{n}(\bdx)|^{p}d\bdx   = 0,
\]
from which the compactness in the $L^{p}_{loc}$ topology  follows.  

We next prove the final conclusion of the     {proposition}.  To that end, let $A \subset \mathbb{R}^{d}$ be a compact subset. 
For $\phi\in C_{c}^{\infty}(B_{1})$, we consider the convoluted sequence of function $ \phi_{\epsilon} *{\bf u}_{n}$, where $\phi_{\epsilon}(\bdz) = \epsilon^{-d} \phi(\bdz/\epsilon)$ is the standard mollifier. 
Since ${\bf u}_{n} \to {\bf u}$ strongly in $L^{p}(A;\mathbb{R}^{d})$ for a fixed $\epsilon > 0$, we have as $n\to\infty$,
\begin{equation}\label{c2-Conv}
\phi_{\epsilon} *{\bf u}_{n} \to \phi_{\epsilon} *{\bf u} \quad \text{in $C^{2}(A;\mathbb{R}^{d})$}. 
\end{equation}
Using Jensen's inequality, we obtain that for any $\epsilon > 0$, and $n$ large, 
\[
\begin{aligned}
\int_{A}\int_{A} \rho_{n}(\bdy-\bdx) &\left|\frac{(\phi_{\epsilon} *{\bf u}_{n} (\bdy) - \phi_{\epsilon} *{\bf u}_{n} (\bdx)) \cdot (\bdy - \bdx)}{ |\bdy - \bdx|^{2}}\right|^{p} d\bdy d\bdx\\
 &\leq \int_{\mathbb{R}^{d}}\int_{\mathbb{R}^{d}} \rho_{n}(\bdy-\bdx) \left| \frac{ ({\bf u}_{n} (\bdy) - {\bf u}_{n} (\bdx)) \cdot (\bdy - \bdx)}{|\bdy - \bdx|^{2}}\right|^{p}d\bdy d\bdx.
\end{aligned}
\]
Taking the limit in $n$ for fixed $\epsilon$, we obtain  for any $A$ compact that 
\[
\int_{A}\int_{A} \rho(\bdy-\bdx) \left| \frac{(\phi_{\epsilon} *{\bf u} (\bdy) - \phi_{\epsilon} *{\bf u} (\bdx)) \cdot (\bdy - \bdx)}{|\bdy - \bdx|^{2}}\right|^{p} d\bdy d\bdx \leq \sup_{n\geq 1} {|{\bf u}_{n}|_{\mathcal{S}_{\rho_{n}, p}}^{p}} < \infty.
\]
where we have used \eqref{c2-Conv} and the fact that $\rho_{n} $ converges weakly to $\rho$ in $L^{1}$.  Finally, let $\epsilon \to 0$ and  use Fatou's lemma (since $\phi_{\epsilon} *{\bf u}\to {\bf u}$ almost everywhere)  to obtain that for any compact set $A$,
\[
\int_{A}\int_{A} \rho(\bdy-\bdx) \left| \frac{({\bf u} (\bdy) - {\bf u} (\bdx)) \cdot (\bdy - \bdx)} {|\bdy - \bdx|^{2}}\right|^{p} d\bdy d\bdx \leq \sup_{n\geq 1} {|{\bf u}_{n}|_{\mathcal{S}_{\rho_{n}, p}}^{p}} < \infty,
\]
hence completing the proof. 
\end{proof}

\section{Global compactness}
\label{sec3}
In this section we prove Theorem \ref{main-compactness}. 
 We follow the approach presented in \cite{Ponce2003}. The argument relies on controlling the $L^{p}$ mass of each ${{\bf u}_{n}}$, $\int_{\Omega} |{\bf u}_{n}| ^{p} d\bdx$, near the boundary by using the bound on the seminorm to demonstrate that  in the limit when $n\to \infty$ there is no mass concentration or loss of mass at the boundary.  This type of control has been done for the sequence of kernels that converge to the Dirac Delta measure in the sense of measures. We will do the same for a fixed locally integrable kernel $\rho$ satisfying the condition \eqref{suff-cond}.

 \subsection{Some technical estimates}
In order to 
control the behavior of functions near the boundary by the semi-norm $|\cdot|_{\mathcal{S}_{p, \rho}}$,
we first present a few technical lemmas.
 \begin{lemma}\cite{Ponce2003}\label{ponce-lemma}
 Suppose that $1\leq p< \infty$ and that  $g\in L^{p}(0, \infty)$. Then there exists a constant $C = C(p)$ such that for any $\delta > 0$ and $t\in (0, \delta)$
 \[
 \int_{0}^{\delta} |g(x)|^{p} dx \leq C \delta^{p} \int_{0}^{2\delta} \frac{|g(x +t )-g(x)|^{p}}{t^{p}} dx  + 2^{p-1}\int_{\delta}^{3\delta} |g(x)|^{p} dx
 \]
 \end{lemma}
 \begin{proof}
 For a given $t\in (0, \delta)$, choose $k$ to be the first positive  integer such that  $kt > \delta$. Observe that $(k-1)t \leq \delta$, and so $kt \leq 2\delta$. 
 Now let us write 
 \[
 \begin{aligned}
 |g(x)|^{p}&\leq 2^{p-1} (|g(x  + kt) - g(x)|^{p} +  |g(x + kt)|^{p} )\\
 &\leq 2^{p-1} k^{p-1} \sum_{j=0}^{k-1}|g(x + jt + t) - g(x    {+ jt})|^{p} + 2^{p-1} |g(x + kt)|^{p}.
 \end{aligned} 
  \] 
  We now integrate in $x$ on both side over $(0, \delta)$ to obtain that 
  \[
 \begin{aligned}
 \int_{0}^{\delta}|g(x)|^{p} dx 
 &\leq 2^{p-1} k^{p-1} \sum_{j=0}^{k-1}\int_{0}^{\delta}|g(x + jt + t) - g(x + jt)|^{p}dx + 2^{p-1} \int_{0}^{\delta}|g(x + kt)|^{p}dx\\
 &\leq 2^{p-1} k^{p-1} \sum_{j=0}^{k-1}\int_{jt}^{\delta + jt}|g(x +  t) - g(x)|^{p}dx + 2^{p-1} \int_{\delta}^{3\delta}|g(x)|^{p}dx\\
 &\leq 2^{p-1} k^{p} \int_{0}^{2\delta }|g(x +  t) - g(x)|^{p}dx + 2^{p-1} \int_{\delta}^{3\delta}|g(x)|^{p}dx\,.
 \end{aligned} 
  \] 
  Recalling that $kt \leq 2\delta$, we have that $k^{p} \leq 2^{p}\delta^{p}/t^{p}$ and we finally obtain the conclusion of the lemma with $C = 2^{2p-1}$. 
\end{proof}

The above lemma will be used on functions of type $t\mapsto {\bf u}(\bdx + t{\bf v})\cdot {\bf v}$, for ${\bf v}\in \mathbb{S}^{d-1}$.  Before doing so, we need to make some preparation first. Observe that since $\Omega$ is a bounded open subset of $\mathbb{R}^{d}$ with a  Lipschitz boundary,
there exist positive constants $r_{0}$ and $\kappa$ with the property that for each point ${\bfs \xi}\in \partial \Omega$ there corresponds a coordinate system $(\bdx',x_{d})$ with $\bdx'\in \mathbb{R}^{d-1}$ and $x_{d}\in\mathbb{R}$ and a Lipschitz continuous function $\zeta:\mathbb{R}^{d-1}\to \mathbb{R}$  such that $|\zeta(\bdx')-\zeta(\bdy')|\leq \kappa|\bdx'-\bdy'|$,
\begin{equation*}
\Omega\cap B({\bfs \xi},4r_{0})= \{(\bdx',x_{d}): x_{d}>\zeta(\bdx')\}\cap B({\bfs \xi}, 4r_{0}),
\end{equation*}
and $\partial\Omega\cap B({\bfs \xi},4r_{0})= \{(\bdx',x_{d}): x_{d}=\zeta(\bdx')\}\cap B({\bfs \xi}, 4r_{0}).$
It is well known that a Lipschitz domain has a uniform interior cone $\Sigma({\bfs \xi}, \theta)$ at every boundary point ${\bfs \xi}$ such that $B({\bfs \xi}, 4\,r_{0})\cap \Sigma({\bfs \xi}, \theta)\subset \Omega.$   The uniform aperture $\theta \in(0, \pi)$ of such cones depends on the Lipschitz constant $\kappa$ of the local defining function $\zeta,$ and does not depend on ${\bfs \xi}.$  It is not difficult either  to see that for any $r\in (0, 4r_{0}),$ if $\bdy \in B_{r}({\bfs \xi}),$ then
\begin{equation*}
\text{dist}(\bdy, \partial \Omega) = \inf\{|{\bdy-(\bdx',x_{d})}|: (\bdx',x_{d})\in B_{3r}({\bfs \xi}), x_{d}=\zeta(\bdx')\}.
\end{equation*}

We now begin to work on local boundary estimates. To do that without loss of generality, after translation and rotation (if necessary) we may assume that ${\bfs \xi} = {\bfs 0}$ and
\begin{equation*}
\Omega\cap B({\bfs 0},4r_{0})= \{(\bdx',x_{d}): x_{d}>\zeta(\bdx')\}\cap B({\bfs 0}, 4r_{0}),
\end{equation*}
where $\zeta({\bfs 0}')=0,$  and $|\zeta(\bdx')-\zeta(\bdy')|\leq \kappa|\bdx'-\bdy'|$. We also assume that the Lipschitz constant $\kappa = 1/2$ and the uniform aperture $\theta = {\pi}/{4}.$
 As a consequence, $\zeta(\bdx')\leq |\bdx'|/2$ for all $\bdx'\in B_{4r_{0}}({\bfs 0}'). $  Given any $0<r<r_{0}, $ we consider the graph of $\zeta:$
\begin{equation*}
\Gamma_{r}:= \{\bdx= (\bdx', \zeta(\bdx'))\in \mathbb{R}^{d}: \bdx'\in B_{r}({\bfs 0}')\}.
\end{equation*}
We denote the upper cone with aperture ${\pi}/{4}$ by $\Sigma$ and is given by
\begin{equation*}
\Sigma = \{\bdx= (\bdx', x_{d})\in \mathbb{R}^{d}: |\bdx'|\leq x_{d}\}.
\end{equation*}
Finally we define $\Omega_{\tau} = \{\bdx\in \Omega: \text{dist}(\bdx, \partial\Omega )>\tau\}$ to be the set of points in $\Omega$ at least $r$ units away from the boundary.
Based on the above discussion we have that for any $r\in (0, r_{0}],$
\begin{equation}\label{inclusion-inequality}
\Omega\cap B_{r/2}\subset \Gamma_{r} + (\Sigma \cap B_{r}) \subset \Omega\cap B_{3r}.
 \end{equation}
 Indeed, let us pick $\bdx = (\bdx', x_{d})\in \Omega\cap B_{r/2}$. The point $\boldsymbol{{\bfs \xi}} = \bdx- (\bdx', \zeta(\bdx')) = ({\bfs 0}', x_{d}-\zeta(\bdx'))\in \Sigma,$ since $   {0<} x_{d}-\zeta(\bdx').$ Moreover, by the bound on the Lipschitz constant $|\boldsymbol{{\bfs \xi}}|=|x_{d}-\zeta(\bdx')|< r/2 + r/4< r.$
On the other hand, for any 
\[
\bdx = (\bdx_{1}', \zeta(\bdx'_{1})) + (\bdx'_{2}, (x_{2})_{d})\in \Gamma_{r} + (   {\Sigma} \cap B_{r}),\]
we have
\[\zeta(\bdx_{1}' + \bdx'_{2})-\zeta(\bdx_{1}') \leq {|\bdx'_{2}|}/{2} \leq {(x_{2})_{d}}/{2},
\] 
showing that $ \zeta(\bdx_{1}' + \bdx'_{2}) < \zeta(\bdx_{1}') + (x_{2})_{d}$ and therefore $\bdx\in \Omega$. It easily follows that $\bdx\in B_{3r},$ as well.

\begin{figure}[htbp] 
\centering
  \begin{tikzpicture}[scale=1.2]  
    \tikzset{to/.style={->,>=stealth',line width=1pt}}; 
        \draw[to] (0,0)--(5,0) node[xshift=.3cm] {$\bdx'$};;
           \draw[to] (0,0)--(0,4) node[yshift=.3cm] {$x_d$};
     \draw[line width=1pt] (-1,1)--(1,2)--(2,1.5)--(4,2.2)--(5,1.5);       
     \node at (3.2, 1.6) {$\zeta(\bdx')$};
         \node at (5.3, 1.6) {$\partial\Omega$};
       \draw (3,3) circle(1.5pt) [fill] ;
    \draw[red, dashed, line width=2pt]   (3,3)--(2.5, 1.7) node[black, yshift=.8cm] {$r$};
      \draw[dashed, line width=2pt]   (3,3)--(3.45, 2.15);
    \draw[blue, line width=1pt] (1, 3.2)--(2.5, 1.7)--(4, 3.2);
    \node at (2.5, 3.4) {$\Sigma$}; 
    
        \draw[red, line width=1pt] (0.5, 0.7)--(2.5, 1.7)--(4.5, 2.7);
      \draw[red, line width=1pt] (0.5, 2.7)--(2.5, 1.7)--(4.5, 0.7);
      
    \node at (5, 2.8) {\footnotesize slope $= 1/2$};
      \node at (5, 0.6) {\footnotesize  slope $= -1/2$};

      \end{tikzpicture}
 \caption{  $\Sigma$ is the cone with aperture $\pi/4$ (depicted by the blue lines). The Lipschitz graph $\zeta$ remains outside the double cone with aperture $\pi/2 - \arctan(1/2)$ (depicted by the red lines). The red dashed line has length $r$.   } 
 \label{fig:Lipgraph}     
\end{figure}
 For any $r\in (0, r_0)$, and $\bdx'\in B_{r} ({\bfs 0}')$, and ${\bf v} \in \Sigma\cap\mathbb{S}^{d-1}$, $\text{dist}((\bdx', \zeta(\bdx')) +r {\bf v},\partial\Omega)\geq r/\sqrt{10}$.
 Indeed, $\text{dist}((\bdx', \zeta(\bdx')) +r {\bf v},\partial\Omega)$ is larger than or equal to the length of the black dashed line, which is larger than or equal to $r\sin(\pi/4-\arctan(1/2)) =r/\sqrt{10}$.

\subsection{Near boundary estimate}
In this subsection we establish the near boundary estimate in the following  
lemma. 
 \begin{lemma}\label{boundary-control-apendix}
 Suppose that $\Omega\subset \mathbb{R}^{d}$ is a domain with Lipschitz boundary. 
 Let $1\leq p<\infty$.   Then there exist positive constants $ C_{1}, C_{2}$, $r_{0}$ and $\epsilon_{0}\in (0, 1)$ with the property that for any $r\in (0, r_{0}), $  ${\bf u}\in L^{p}(\Omega;\mathbb{R}^{d}),$ and any nonnegative and nonzero $\rho\in L^{1}_{loc}(\mathbb{R}^{d})$ that is radial, we have 
 \begin{equation*}
 \int_{\Omega}|{\bf u}|^{p}d\bdx  \leq C_{1}(r) \int_{\Omega_{\epsilon_{0}r}}|{\bf u}|^{p}d\bdx + {C_{2}  \frac{r^{p}}{\displaystyle
 \int_{B_{r}(0)}\rho({\bf h}) d{\bf h}}} \int_{\Omega}\int_{\Omega}\rho(\bdx-\bdy)\left|\mcD({\bf u})(\bdx,\bdy)\right|^{p} d\bdx\,d\bdy.
 \end{equation*}
 The constant $C_{1}$ may depend on $r$ but the other constants  $C_{2}$ and $r_{0}$ depend only on $d, p$ and the Lipschitz constant of $ \Omega$.  Here for any $\tau > 0$, we define $\Omega_{\tau} = \{\bdx\in \Omega: \text{dist}(\bdx, \partial\Omega )>\tau\}$. 
\end{lemma}
\begin{proof}
Following the above discussion, let us pick $\overline{\bfs \eta}\in \partial \Omega$ and assume without loss of generality that $\overline{\bfs \eta} = {\bfs 0},$ the function $\zeta$ that defines the boundary $\partial \Omega$ has a Lipschitz constant not bigger than $1/2$ and the aperture is $\pi/4.$
 
 Assume first that ${\bf u} \in L^{p}(\Omega;\mathbb{R}^{d})$, and vanishes on $\Omega_{r/\sqrt{10}}$. 
 Let us pick $\boldsymbol{{\bfs \xi}} = (\bdx', \zeta(\bdx'))$ such that $|\bdx'| < r$ and ${\bf v}\in \Sigma \cap \mathbb{S}^{d-1}.$
 Let us introduce the function
 \begin{equation*}
 g_{\bf v}^{{\bfs \xi}}(t) = {\bf u}({\bfs \xi}+ t{\bf v})\cdot {\bf v}, \quad t\in (0, 3 r_{0}).
 \end{equation*}
 Then for all ${\bfs \xi}\in \Gamma_{r}$ and ${\bf v}\in \Sigma\cap\mathbb{S}^{d-1}$, ${\bfs \xi} + r{\bf v}\in \Omega_{r/\sqrt{10}}$. It follows that, by assumption on the vector field ${\bf u},$ the function  $g_{\bf v}^{{\bfs {\bfs \xi}}}(t) \in L^{p}(0, 2r)$ and $g_{\bf v}^{\bfs {\bfs \xi}}(t) = 0$ for $t\in (r, 2r).$
 We then apply Lemma \ref{ponce-lemma}  
 to get a constant $C_{p}>0$ such that for any
$t\in(0, r)$,
 \begin{equation*}
 \int_{0}^{r}|{\bf u}({\bfs {\bfs \xi}}+ s{\bf v})\cdot {\bf v}|^{p}\,ds \leq C \,r^{p}\int_{0}^{r} \frac{|({\bf u}({\bfs {\bfs \xi}}+ s{\bf v} +t{\bf v} )-{\bf u}({\bfs {\bfs \xi}}+ s{\bf v}))\cdot {\bf v}|^{p}}{t^{p}}ds, 
 \end{equation*}
 where we used the fact that $u$ vanishes on $\Omega_{r/\sqrt{10}}$. 
 Noting that $\boldsymbol{{\bfs \xi}} = (\bdx', \zeta(\bdx'))$ for some $\bdx'\in B'_{r}\subset \mathbb{R}^{d-1},$ we integrate first in the above estimate with respect to $\bdx'\in B'_{r}$ to obtain that
 \begin{equation*}
 \begin{aligned}
 & \int_{B'_{r}}\int_{0}^{r}|{\bf u}({\bfs \xi}+ s{\bf v})\cdot {\bf v}|^{p}\,ds\,d\bdx' \\
 &\qquad \leq C \,r^{p}\int_{B'_{r}}\int_{0}^{r} \frac{|({\bf u}({\bfs \xi}+ s{\bf v} +t{\bf v} )-{\bf u}({\bfs \xi}+ s{\bf v}))\cdot {\bf v}|^{p}}{t^{p}}ds\,d\bdx' 
 \end{aligned}
 \end{equation*}
% By making a nonlinear change of variables $\bdy = (\bdx', \zeta(\bdx')) + s{\bf v}$, we note that the Jacobian of this map is 1 and maps the cylinder $B'_{\tau}\times[\tau_{1}, \tau_{2}]$ to $\Gamma_{\tau} + \Sigma\cap \left(B_{\tau_{2}}\setminus B_{\tau_{1}}\right)$.   As a consequence, for all ${\bf v}\in \Sigma\cap\mathbb{S}^{d-1}$,
 {The next step involves a change of variable $\bdy = (\bdx', \zeta(\bdx')) + s{\bf v}$. Define a mapping $G: (\bdx', s) \mapsto (\bdx', \zeta(\bdx')) + s{\bf v}$. Then the Jacobian of the mapping is defined almost everywhere and is given by \[
 J_G = {\bf v} \cdot (-\nabla \zeta^T(\bdx'), 1). 
 \]
 Notice that $|J_G|$ is bounded from above and below by two constants since ${\bf v}\in \Sigma\cap\mathbb{S}^{d-1}$ and the Lipschitz constant of $\zeta$ is not bigger than $1/2$. Also notice that $G(B_r'\times (0,1)) \subset \Gamma_{r} + \Sigma\cap B_{r}$. Therefore,
 \begin{equation*}
 \begin{aligned}
 & \int_{B'_{r}}\int_{0}^{r}  \frac{|({\bf u}({\bfs \xi}+ s{\bf v} +t{\bf v} )-{\bf u}({\bfs \xi}+ s{\bf v}))\cdot {\bf v}|^{p}}{t^{p}}ds\,d\bdx' \\
 &\leq C  \int_{\Gamma_{r} + \Sigma\cap B_{r}}
 \frac{|({\bf u}(\bdy +t{\bf v} )-{\bf u}(\bdy))\cdot {\bf v}|^{p}}{t^{p}}d\bdy\\
 &\leq C  \int_{\Omega\cap B_{3r}}
 \frac{|({\bf u}(\bdy +t{\bf v} )-{\bf u}(\bdy))\cdot {\bf v}|^{p}}{t^{p}}d\bdy,
 \end{aligned}
 \end{equation*}
 where we have used \eqref{inclusion-inequality} in the last step. By some straightforward calculations, one can also find that $\Omega\cap B_{r/4} \subset G(B_r'\times (0,1))$.   
Then 
\[
\int_{\Omega\cap B_{\frac{r}{4}}}|{\bf u}(\bdy)\cdot {\bf v}|^{p}d\bdy \leq \int_{G(B_r'\times (0,1))}|{\bf u}(\bdy)\cdot {\bf v}|^{p}d\bdy \leq C \int_{B'_{r}}\int_{0}^{r}|{\bf u}({\bfs \xi}+ s{\bf v})\cdot {\bf v}|^{p}\,ds\,d\bdx'\,.
\]
 }
 It then follows from the above calculations that
 that for 
 all ${\bf v}\in \Sigma\cap\mathbb{S}^{d-1}$ and all $t\in (0,r),$
 \begin{equation}\label{local-bry-inequality}
 \int_{\Omega\cap B_{ \frac{r}{4}}}|{\bf u}(\bdy)\cdot {\bf v}|^{p}d\bdy \leq C \, r^p\int_{\Omega\cap B_{3r}} \frac{|({\bf u}(\bdy +t{\bf v} )-{\bf u}(\bdy))\cdot {\bf v}|^{p}}{t^{p}}d\bdy. 
 \end{equation}
Multiplying the left hand side of \eqref{local-bry-inequality} by $\rho(t{\bf v})t^{d-1}$ and integrating in $t\in(0, r)$ and in ${\bf v}\in \Sigma\cap \mathbb{S}^{d-1}$, we get
 \begin{equation*}
 \begin{aligned}
 &\int_{0}^{r}\int_{\Sigma\cap \mathbb{S}^{d-1}} \int_{\Omega\cap B_{\frac{r}{4}}}|{\bf u}(\bdy)\cdot {\bf v}|^{p}\rho(t{\bf v}) t^{d-1}d\bdy\, d\sigma({\bf v}) \,dt\\
 &\qquad = \int_{\Omega\cap B_{\frac{r}{4}}}\int_{\Sigma\cap B_{r}}|{\bf u}(\bdy)\cdot \frac{{\bf z}}{|\bdz|}|^{p}\rho(\bdz)d\bdz \,d\bdy \,.
 \end{aligned}
 \end{equation*}
 Using Lemma \ref{positive-inf-on-sector}, we observe that
 \begin{equation}\label{esti-for-u}
 \begin{aligned}
&  \int_{\Omega\cap B_{\frac{r}{4}}}\int_{\Sigma\cap B_{r}}|{\bf u}(\bdy)\cdot \frac{{\bf z}}{|\bdz|}|^{p}\rho(\bdz)d\bdz \,d\bdy\\
 &\qquad
 = \int_{\Omega\cap B_{\frac{r}{4}}}|{\bf u}(\bdy) |^{p}\int_{\Sigma\cap B_{r}}\left|\frac{{\bf u}(\bdy)}{|{\bf u}(\bdy)|}\cdot \frac{{\bf z}}{|\bdz|}\right|^{p}\rho(\bdz)d\bdz \,d\bdy\\
 &\qquad \geq \left(\int_{0}^{r}t^{d-1}\rho(t) dt\right) \int_{\Omega\cap B_{{r \over 4}}}|{\bf u}(\bdy) |^{p} \int_{\Sigma\cap \mathcal{S}^{d-1}}\left|\frac{{\bf u}(\bdy)}{|{\bf u}(\bdy)|}\cdot {\bf w}\right|^{p}d \mathcal{H}^{d-1}({\bf w}) \,d\bdy \\
 &\qquad \geq c_{0} \left(\int_{B_{r}}\rho(   {{\bfs \xi}})d   {{\bfs \xi}}\right) \int_{\Omega\cap B_{\frac{r}{4}}}|{\bf u}(\bdy)|^{p}d\bdy \,.
 \end{aligned}
 \end{equation}
 Similarly, we have 
 \begin{equation}\label{esti-for-Du}
  \begin{aligned}
 \int_{0}^{r}\int_{\Sigma\cap \mathbb{S}^{d-1}}& \int_{\Omega\cap B_{3r}} \frac{|({\bf u}(\bdy +t{\bf v} )-{\bf u}(\bdy))\cdot {\bf v}|^{p}}{t^{p}}\rho(t{\bf v}) t^{d-1}d\bdy\,d\sigma({\bf v})\,d t\\
 &= \int_{\Omega\cap B_{3r}}\int_{\Sigma\cap B_{r}}\frac{|({\bf u}(\bdy +{\bf z} )-{\bf u}(\bdy))\cdot \frac{{\bf z}}{|\bdz|}|^{p}}{|\bdz|^{p}}\rho(|{\bf z}|) d{\bf z} d\bdy\,\\
 &\leq \int_{\Omega\cap B_{4r}}\int_{\Omega\cap B_{4r}}|\mcD({\bf u})(\bdx,\bdy)|^p\rho(\bdx-\bdy)d\bdy\, d{\bf x}.
  \end{aligned}
 \end{equation}
Combining inequalities \eqref{esti-for-u}, \eqref{esti-for-Du} and  \eqref{local-bry-inequality} we obtain that 
 \begin{equation}\label{est-at-bdry-point}
 c_{0} \int_{\Omega\cap B_{\frac{r}{4}}}|{\bf u}(\bdy)|^{p}d\bdy \leq \,\frac{r^{p}}{\displaystyle
 \int_{B_r} \rho(   {{\bfs \xi}})d   {{\bfs \xi}} }\,\int_{\Omega\cap B_{4r}}\int_{\Omega\cap B_{4r}}|\mcD({\bf u})(\bdx,\bdy)|^p\rho(\bdx-\bdy)d\bdy\, d{\bf x} 
 \end{equation}
 for some positive constant $c_{0}$ which only depends on $d, p$ and the Lipschitz constant of the domain. In particular, 
 the estimate \eqref{est-at-bdry-point} holds true at all boundary points $   {\overline{\bfs \eta}} \in \partial \Omega$. 

The next argument is used in the proof of \cite[Lemma 5.1]{Ponce2003}. By applying standard covering argument, it follows from 
the inequality \eqref{est-at-bdry-point} that there exist positive constants $\epsilon_{0}\in (0, 1/(2\sqrt{10}))$ and $C$ with the property that for all $r\in(0, r_{0}),$   such that for all ${\bf u}\in L^{p}(\Omega;\mathbb{R}^{d})$ that vanishes in $\Omega_{{r}/{\sqrt{10}}}$ 
\begin{equation}\label{lp-boundary}
\int_{\Omega\setminus\Omega_{2\epsilon_{0}r }} |{\bf u}|^{p}d\bdx\leq  C \,\frac{r^{p}}{\displaystyle
\int_{B_r} \rho(   {{\bfs \xi}})d   {{\bfs \xi}} }\int_{\Omega}\int_{\Omega} |\mcD({\bf u})(\bdx,\bdy)|^p\rho(\bdx-\bdy)d\bdy\, d{\bf x}.
\end{equation}
The positive constants $\epsilon_{0}$ and $C$ depend only on $p$ and the Lipschitz character of the boundary of $\Omega.$ For ease of calculation, set $\tilde{r} = 2r/\sqrt{10}$. Then $\tilde{r}/2 = r/\sqrt{10}$. 

Now let ${\bf u}\in L^{p}(\Omega;\mathbb{R}^{d}),$ and let $\phi\in C^{\infty}(\Omega)$ be such that: $\phi(\bdx)=0$, if $\bdx\in \Omega_{\tilde{r}/2}$; $0\leq \phi(\bdx)\leq 1,$ if $\bdx\in \Omega_{\tilde{r}/4}\setminus \Omega_{\tilde{r}/2}$; $\phi(\bdx)=1,$ if $\bdx\in \Omega\setminus \Omega_{\tilde{r}/4}$
and $|\nabla \phi|\leq C/{r}$ on $\Omega.$
Applying \eqref{lp-boundary} to the vector field $\phi(\bdx){\bf u}(\bdx)$, we obtain that 
\begin{equation*}
\begin{aligned}
\int_{\Omega\setminus \Omega_{\epsilon_{0} \,r }} |{\bf u}|^{p}d\bdx &\leq C\frac{r^{p}}{\displaystyle
 \int_{B_r} \rho(   {{\bfs \xi}})d   {{\bfs \xi}} }\int_{\Omega}\int_{\Omega} |\mcD(\phi{\bf u})(\bdx,\bdy)|^p\rho(\bdx-\bdy)d\bdy\, d{\bf x}
\end{aligned}
\end{equation*}
We may rewrite    {$\mcD(\phi{\bf u})$} as follows 
\[
\mcD(\phi{\bf u})(\bdx,\bdy) = (\phi(\bdx) + \phi(\bdy))\mcD({\bf u})(\bdx,\bdy) - \left({\phi(\bdx){\bf u}(\bdy) - \phi(\bdy){\bf u}(\bdx)\over |\bdy-\bdx|}\right)\cdot {(\bdy-\bdx)\over |\bdy-\bdx|}. 
\]
It then follows that 
\[
\begin{aligned}
&
\int_{\Omega}\int_{\Omega} |\mcD(\phi{\bf u})(\bdx,\bdy)|^p\rho(\bdx-\bdy)d\bdy\, d{\bf x} \\
&\quad \leq 2^{p-1} \int_{\Omega}\int_{\Omega} |[\phi(\bdx) + \phi(\bdy)]\mcD({\bf u})(\bdx,\bdy)|^p\rho(\bdx-\bdy)d\bdy\, d{\bf x} \\ 
&\qquad + 2^{p-1}\int_{\Omega}\int_{\Omega} \left|{\phi(\bdx){\bf u}(\bdy) - \phi(\bdy){\bf u}(\bdx)\over |\bdy-\bdx|}\cdot {(\bdy-\bdx)\over |\bdy-\bdx|}\right|^p\rho(\bdx-\bdy)d\bdy\, d{\bf x}\\
&\quad =2^{p-1} \left(I_{1} + I_{2}\right).
\end{aligned}
\]
The first term $I_{1}$ can be easily estimated as 
\[\begin{aligned}
I_{1} &= \int_{\Omega}\int_{\Omega} |[\phi(\bdx) + \phi(\bdy)]\mcD({\bf u})(\bdx,\bdy)|^p\rho(\bdx-\bdy)d\bdy\, d{\bf x}\\
& \leq 2 \int_{\Omega}\int_{\Omega} |\mcD({\bf u})(\bdx,\bdy)|^p\rho(\bdx-\bdy)d\bdy\, d{\bf x}.
\end{aligned}
\]
Let us estimate the second term, $I_{2}$. We first break it into three integrals. 
\[
\begin{aligned}
I_{2} &= 
\int_{\Omega}\int_{\Omega} \left|{\phi(\bdx){\bf u}(\bdy) - \phi(\bdy){\bf u}(\bdx)\over |\bdy-\bdx|}\cdot {(\bdy-\bdx)\over |\bdy-\bdx|}\right|^p\rho(\bdx-\bdy)d\bdy\, d{\bf x}\\
& = \iint_{A} + \iint_{B} +  \iint_{C}
\end{aligned}
\]
where $A = \Omega\setminus \Omega_{\tilde{r}/4} \times \Omega\setminus \Omega_{\tilde{r}/4}$, $B=(\Omega\setminus \Omega_{\tilde{r}/8})\times \Omega_{\tilde{r}/4}  \cup  \left(\Omega_{\tilde{r}/4} \times (\Omega\setminus \Omega_{\tilde{r}/8})\right)$ and $C = \Omega\times\Omega \setminus (A\cup B)$. 
We estimate each of these integrals. Let us begin with the simple one: $ \iint_{A}$. After observing that $\phi(\bdx)=\phi(\bdy) =1$ for all $\bdx, \bdy\in \Omega\setminus \Omega_{\tilde{r}/4},$ we have that
\[
\iint_{A} = \int_{\Omega\setminus \Omega_{\tilde{r}/4}}\int_{\Omega\setminus \Omega_{\tilde{r}/4}}  |\mcD({\bf u})(\bdx,\bdy)|^p\rho(\bdx-\bdy)d\bdy\, d{\bf x}, 
\]
and the    {latter} is bounded by the semi norm. 
Next, we note that set $B$ is symmetric with respect to the diagonal, and as a result, 
\[
\iint_{B}  = 2 \int_{\Omega\setminus \Omega_{\tilde{r}/8}} \int_{\Omega_{\tilde{r}/4}}
\]
and when $(\bdx, \bdy) \in (\Omega\setminus \Omega_{\tilde{r}/8})\times \Omega_{\tilde{r}/4} $, we    {have} $\phi(\bdx) = 1$, and so we have 
\[
\begin{aligned}
\iint_{B} &= 2\int_{\Omega\setminus \Omega_{\tilde{r}/8}} \int_{\Omega_{\tilde{r}/4}} \left|{{\bf u}(\bdy) - \phi(\bdy){\bf u}(\bdx)\over |\bdy-\bdx|}\cdot {(\bdy-\bdx)\over |\bdy-\bdx|}\right|^p\rho(\bdx-\bdy)d\bdy\, d{\bf x}\\
&\leq 
2^{p}\int_{\Omega\setminus \Omega_{\tilde{r}/8}} \int_{\Omega_{\tilde{r}/4}}  \left|\phi(\bdy)\mcD({\bf u})(\bdx,\bdy)\right|^p\rho(\bdx-\bdy)d\bdy\, d{\bf x}\\
&\qquad  + 2^{p}\int_{\Omega\setminus \Omega_{\tilde{r}/8}} \int_{\Omega_{\tilde{r}/4}} \left|{{\bf u}(\bdy)\over |\bdy-\bdx|}\right|^p\rho(\bdx-\bdy)d\bdy\, d{\bf x}\\
&\leq 2^{p}\int_{\Omega\setminus \Omega_{\tilde{r}/8}} \int_{\Omega_{\tilde{r}/4}}\left|\mcD({\bf u})(\bdx,\bdy)\right|^p\rho(\bdx-\bdy)d\bdy\, d{\bf x}\\
&\qquad  + 
\frac{2^{4p}}{r^{p}}\int_{\Omega\setminus \Omega_{\tilde{r}/8}} \int_{\Omega_{\tilde{r}/4}} |{\bf u}(\bdy)|^p\rho(\bdx-\bdy)d\bdy\, d{\bf x}\\
\end{aligned}
\]
where we have used the fact that 
$\text{dist}(\Omega\setminus \Omega_{\tilde{r}/8}, \Omega_{\tilde{r}/4})=\tilde{r}/{8}$. 
As a consequence we have that 
\[
\iint_{B} \leq 2^{p} \int_{\Omega}\int_{\Omega} |\mcD({\bf u})(\bdx,\bdy)|^p \rho(\bdy-\bdx)d\bdy\,d\bdx +  \frac{2^{4p}}{r^{p}} \left(\int_{|{\bf h}| > \tilde{r}/{8}} \rho({\bf h}) d{\bf h} \right) \int_{\Omega_{\tilde{r}/4}} |{\bf u}(\bdy)|^{p} d\bdy.
\]
To estimate the integral on $C$, we first observe that  for any $(\bdx, \bdy) \in C$, then $\text{dist}(\bdx, \partial \Omega)\geq \tilde{r}/{8}$ and $\text{dist}(\bdy, \partial \Omega)\geq \tilde{r}/{8}$. Using this information, adding and subtracting $\phi(\bdx){\bf u}(\bdx)$ we can then estimate as follows:
\[
\begin{aligned}
\iint_{C} &\leq 2^{p-1} \iint_{C} |\mcD({\bf u})(\bdx,\bdy)|^p \rho(\bdy-\bdx)d\bdy\,d\bdx \\
&\qquad
+ 2^{p-1}\iint_{C}|{\bf u}(\bdx)|^{p} \frac{|\phi(\bdx )-\phi(\bdy)|^{p}}{|\bdx-\bdy|^{p}}\rho(|\bdx-\bdy|)d\bdy\, d{\bf x} \\
& \leq 2^{p-1} \int_{\Omega}\int_{\Omega} |\mcD({\bf u})(\bdx,\bdy)|^p \rho(\bdy-\bdx)d\bdy\,d\bdx\\
& \qquad +  \frac{C}{r^{p}} \int_{B_{R}} \rho({\bf h}) d{\bf h} \int_{\Omega_{{r\over8}}}|{\bf u}(\bdx)|^{p} d\bdx
\end{aligned}
\]
where we used the estimate $|\nabla \phi| \leq {C\over r}$, and denoted $R = \text{diam}(\Omega)$. 

We then conclude that there exists a universal constant $C > 0$ such that for any $r$ small
\[
\begin{aligned}
\int_{\Omega\setminus \Omega_{\epsilon_{0} \,r }} |{\bf u}|^{p}d\bdx &
\leq C \left({r^{p}\over \displaystyle
 \int_{B_{r}}\rho(|{\bdy}|)d\bdy} \int_{\Omega}\int_{\Omega} |\mcD({\bf u})(\bdx,\bdy)|^p \rho(\bdy-\bdx)d\bdy\,d\bdx \right.\\
&\qquad
+\left.  \frac{1}{r^{p}}\int_{B_{R}} \rho({\bf h}) d{\bf h}  \int_{\Omega_{\tilde{r}/{8}}} |{\bf u}|^{p} d\bdx\right) .
\end{aligned}
\]
It then follows that 
\begin{equation*}
\begin{aligned}
\int_{\Omega}|{\bf u}|^{p}d\bdx & = \int_{\Omega_{\epsilon_{0}r}}|{\bf u}|^{p}d\bdx + \int_{\Omega\setminus \Omega_{\epsilon_{0}r}}|{\bf u}|^{p}d\bdx \\
&\leq \int_{\Omega_{\epsilon_{0} r}}|{\bf u}|^{p}d\bdx  + C\,{r^{p}\over \displaystyle
\int_{B_{r}}\rho(|{\bdy}|)d\bdy}  \int_{\Omega}\int_{\Omega} |\mcD({\bf u})(\bdx,\bdy)|^p \rho(\bdy-\bdx)d\bdy\,d\bdx \\
&\qquad+ C {\|\rho\|_{L^{1}(B_{R})}\over r^{p}}\int_{\Omega_{\tilde{r}/8}}|{\bf u}|^{p}d\bdx.
\end{aligned}
\end{equation*}
We hence complete the proof of Lemma \ref{boundary-control-apendix}
after choosing $\epsilon_{0}$ sufficiently small, say for example $\epsilon_{0}<{1}/{4\sqrt{10}}$, that
\begin{equation*}
\int_{\Omega}|{\bf u}|^{p}d\bdx \leq C(r)\int_{\Omega_{\epsilon_{0}r}}|{\bf u}|^{p}d\bdx + C\,{r^{p}\over 
\displaystyle \int_{B_{r}(0)}\rho(|{\bdy}|)d\bdy} \int_{\Omega}\int_{\Omega} |\mcD({\bf u})(\bdx,\bdy)|^p \rho(\bdy-\bdx)d\bdy\,d\bdx,
\end{equation*}
as desired.
%%%%%
%%%%%

\subsection{Compactness in $L^{p}(\Omega)$: proof of Theorem \ref{main-compactness}}
Let ${\bf u}_{n}$ be a bounded sequence in $\mathcal{S}_{\rho,p} (\Omega)$. Let $\phi_{j} \in C^{\infty}_{0}(\Omega)$ such that $\phi_{j} \equiv 1$ in $\Omega_{1/j}$. Then the sequence $\{\phi_{j}{\bf u}_{n}\}_{n}$ is bounded in  $\mathcal{S}_{\rho,p} (\mathbb{R}^{d})$, and so by Theorem \ref{loc-compactness-main}, $\phi_{j}{\bf u}_{n}$ is precompact in $\Omega$. In particular, $\{{\bf u}_{n}\}$ is relatively compact in $L^{p}(\Omega_{j})$. From this one can extract a subsequence ${\bf u}_{n_{j}}$ such that ${\bf u}_{n_{j}} \to {\bf u}$ in $L^{p}_{loc}(\Omega)$. It is easy to see that ${\bf u}\in L^{p}(\Omega)$. In fact, using the pointwise convergence and Fatou's lemma, we can see that ${\bf u}\in \mathcal{S}_{\rho, p}(\Omega)$. What remains is to show that ${{\bf u}_{n_{j}}} \to {\bf u}$ in $L^{p}(\Omega)$.    To that end, we apply Lemma \ref{boundary-control-apendix} for the function ${\bf u}_{n_{j}} - {\bf u}$, to obtain that 
\[
\int_{\Omega}|{\bf u}_{n_{j}} - {\bf u}|^{p}d\bdx \leq C_{1}(r) \int_{\Omega_{\epsilon_{0}r}}|{\bf u}_{n_{j}} - {\bf u}|d\bdx + C_{2} \frac{r^{p}}{\displaystyle  \int_{B_{r}}\rho({\bf h}) d{\bf h} } |{\bf u}_{n_{j}} - {\bf u}|^{p}_{\mathcal{S}_{\rho, p}(\Omega)} 
\]
for all small $r$.
We now fix $r$ and let $j\to \infty$  to obtain that 
\[
\limsup_{j\to \infty}\int_{\Omega}|{\bf u}_{n_{j}} - {\bf u}|^{p}d\bdx \leq  C \frac{r^{p}}{\displaystyle  \int_{B_{r}}\rho({\bf h}) d{\bf h} } (1 + |{\bf u}|^{p}_{\mathcal{S}_{\rho, p}}). 
\]
We then let $r\to 0$, to obtain that $\limsup_{j\to \infty}\int_{\Omega}|{\bf u}_{n_{j}} - {\bf u}|^{p}d\bdx = 0.$
\end{proof}

\subsection{Compactness for a sequence of kernels: proof of Theorem \ref{compactness-sequence}}
Arguing as above and by Proposition \ref{compactness-sequence-loc}, we have that there is a subsequence ${\bf u}_{n_{j}} \to u$ in $L^{p}_{loc}(\Omega)$, and that $u\in \mathcal{S}_{\rho, p}(\Omega)$. To conclude, we apply Lemma \ref{boundary-control-apendix} for the function ${\bf u}_{n_{j}} - {\bf u}$ corresponding to $\rho_{n_{j}}$ to obtain
\[
\int_{\Omega}|{\bf u}_{n_{j}} - {\bf u}|^{p}d\bdx \leq C_{1}(r) \int_{\Omega_{\epsilon_{0}r}}|{\bf u}_{n_{j}} - {\bf u}|d\bdx + C_{2} \frac{r^{p}}{\displaystyle  \int_{B_{r}}\rho_{n_{j}}({\bf h}) d{\bf h} } |{\bf u}_{n_{j}} - {\bf u}|^{p}_{\mathcal{S}_{\rho_{n_j}, p}(\Omega)} 
\]
By assumption $\rho_{n_j}\leq C \rho$ and so $|{\bf u}_{n_{j}} - {\bf u}|^{p}_{\mathcal{S}_{\rho_{n_j}, p}(\Omega)}\leq C |{\bf u}_{n_{j}} - {\bf u}|^{p}_{\mathcal{S}_{\rho, p}(\Omega)}$. We then let 
 $j\to \infty$ and apply the weak convergence of $\rho_{n}$ to obtain that \[
\limsup_{j\to \infty}\int_{\Omega}|{\bf u}_{n_{j}} - {\bf u}|^{p}d\bdx \leq  C \frac{r^{p}}{\displaystyle  \int_{B_{r}}\rho({\bf h}) d{\bf h} } (1 + |{\bf u}|^{p}_{\mathcal{S}_{\rho, p}}). \]
Finally, we let $r\to 0$ to conclude the proof. 

\subsection{Poincar\'e-Korn type inequality: proof of Corollary \ref{Poincare}}
We recall that given $V\subset L^{p}(\Omega;\mathbb{R}^{d})$ satisfying the hypothesis of the corollary, there exists a constant $P_{0}$ such that for any ${\bf u}\in V,$
\begin{equation}\label{limit-poincare}
 \int_{\Omega} |{\bf u}|^{p}d\bdx \leq P_{0} \int_{\Omega}\int_{\Omega}  \rho(\bdy - \bdx)\left|\frac{({\bf u}(\bdy) - {\bf u}({\bdx}))}{|\bdy-\bdx|} \cdot \frac{(\bdy -\bdx)}{|\bdy - \bdx|}\right|^{p}d\bdy d\bdx.
\end{equation}
This result is proved in \cite{Du-Navier1} or \cite{Mengesha-Du}. 
We take $P_{0}$ to be the best constant.  We claim that given any $\epsilon > 0$, there exists $N=N(\epsilon) \in \mathbb{N}$ such that for all $n\geq N$, \eqref{seq-poincare} holds for $C = P_{0} + \epsilon$. We prove this by contradiction. Assume otherwise and that there exists $C > P_{0}$ such that for every $n$, there exists ${\bf u}_{n}\in V\cap L^{p}(\Omega;\mathbb{R}^{d})$, $\|{\bf u}_{n}\|_{L^{p}} =1$, and 
\[
\int_{\Omega}\int_{\Omega}  \rho_{n}(\bdy - \bdx)\left|\frac{({\bf u}_{n}(\bdy) - {\bf u}_{n}({\bdx}))}{|\bdy-\bdx|} \cdot \frac{(\bdy -\bdx)}{|\bdy - \bdx|}\right|^{p}d\bdy d\bdx < \frac{1}{C}.
\] 
By Theorem \ref{compactness-sequence}, ${\bf u}_{n}$ is precompact in $L^{p}(\Omega;\mathbb{R}^{d})$ and therefore any limit point ${\bf u}$  will have $\|{\bf u}\|_{L^{p}} = 1$, and will be in $V\cap L^{p}(\Omega;\mathbb{R}^{d})$. Moreover, following the same procedure as in the proof of     {Proposition} \ref{compactness-sequence-loc}, we obtain that 
\[
\begin{aligned}
\int_{\Omega}\int_{\Omega}  &\rho(\bdy - \bdx)\left|\frac{({\bf u}(\bdy) - {\bf u}({\bdx}))}{|\bdy-\bdx|} \cdot \frac{(\bdy -\bdx)}{|\bdy - \bdx|}\right|^{p}d\bdy d\bdx \\
&\leq \liminf_{n\to \infty}\int_{\Omega}\int_{\Omega}  \rho_{n}(\bdy - \bdx)\left|\frac{({\bf u}_{n}(\bdy) - {\bf u}_{n}({\bdx}))}{|\bdy-\bdx|} \cdot \frac{(\bdy -\bdx)}{|\bdy - \bdx|}\right|^{p}d\bdy d\bdx \leq {1 \over C } < {1\over P_{0}}
\end{aligned}
\]
which gives the desired contradiction since $P_{0}$ is the best constant in \eqref{limit-poincare}. 

\section{Discussions}
In this work, we have presented a set of sufficient conditions that guarantee a compact inclusion of a set of $L^{p}$-vector fields in the Banach space of $L^{p}$ vector fields.  The criteria are nonlocal and given with respect to nonlocal interaction kernels that may not be necessarily radially symmetric. 
 We note that, in addition to the mathematical generality, relaxing the radial symmetry assumption on nonlocal interactions can be useful
 when modeling anisotropic behavior and directional transport.
 The $L^{p}$-compactness is established for a sequence of vector fields where the nonlocal interactions involve only part of their components, so that the results and discussions
represent a significant departure from those known for scalar fields.  
It is not clear yet 
{to what extent the conditions assumed here can be weakened to reach the same conclusions.}
In this regard, there are still some outstanding questions in relation to the set of minimal conditions on the interaction kernel as well as on the set of vector fields that imply $L^{p}$-compactness.  An application of the compactness result that will be explored elsewhere includes designing approximation schemes for nonlocal system of equations of peridynamic-type similar to the one done in \cite{Tian-Du} for nonlocal equations.  
\section{Acknowledgments}

Q. Du's research is partially supported by US National Science Foundation grant     {DMS-1937254 and DMS-2309245}. 

T. Mengesha's research is supported by US National Science Foundation grant DMS-1910180 and DMS-2206252. 

X. Tian's research is supported by US National Science Foundation grant DMS-2111608 and DMS-2240180. 

The authors thank Zhaolong Han for helping improve the proof of Lemma 4.2. 
\appendix
\section{Compactness in $L^{p}_{loc}$ topology}
The following theorem as well as the proof we present here is inspired by the compactness result proved in \cite{Jarlocalcompactness} (see also \cite{Guy-thesis}) for scalar functions which uses the more flexible nonintegrability condition of the ${\rho({\bf z}) \over |{\bf z}|^{p}}$ than the one stated in \eqref{suff-cond}.
\begin{theorem}[$L^{p}_{loc}$ compactness]\label{loc-compactness-S}Suppose that $1\leq p  < \infty$. 
Let $\rho \in L^{1}(\mathbb{R}^{d})$ be a nonnegative radial function satisfying
\begin{equation}\label{suff-loc-compactness}
\lim_{\delta\to 0 } \int_{|{\bf z}| > \delta} {\rho({\bf z})\over |{\bf z}|^{p}}d{\bf z} =\infty. 
\end{equation}
 Suppose also that $\{{\bf u}_{n}\}$ is a sequence of vector fields that is bounded in $\mathcal{S}_{\rho,p}(\mathbb{R}^{d})$. 
Then for any $D\subset \mathbb{R}^{d}$ open and bounded,  the sequence $\{{\bf u}_{n}|_{D}\}$ is precompact in $L^{p}(D; \mathbb{R}^{d})$. 
\end{theorem}
  
As stated earlier, for radial functions with compact support, condition \eqref{suff-loc-compactness} is weaker than \eqref{suff-cond}. 
  Indeed, \eqref{suff-cond} implies that $\rho(\bdz) |\bdz|^{-p} $ is not integrable near ${\bfs 0}$ which implies \eqref{suff-loc-compactness}. One the other hand, the kernel $\rho({\bf z})= |{\bf z}|^{-d-p} \chi_{B_{1}({\bfs 0})}({\bf z})$ satisfies \eqref{suff-loc-compactness} but not \eqref{suff-cond}. 
 
Similar to the argument we gave in section 2, the proof of the theorem will make use of the following  variant of the Riesz-Fr\'echet-Kolomogorov theorem \cite{anotherlook,Mengesha}. 
\begin{lemma}(\cite[Lemma 5.4]{Mengesha}) \label{RFK-berzis}
Let the sequence $\{\mathbb{G}^\delta\}_{\delta>0}\subset L^1(\mathbb{R}^{d}; \mathbb{R}^{d\times d})$ be an approximation to the identity. That is 
\[
\forall \delta>0, \int_{\mathbb{R}^{d}} \mathbb{G}^\delta (\bdx)d\bdx = \mathbb{I}_d,\quad \text{for any $r>0$, } \lim_{\delta \to 0}\int_{|\bdx|>r} \mathbb{G}^\delta (\bdx)d\bdx  = \mathbf{0}
\]
If $\{{\bf f}_n\}_{n}$ is a bounded sequence in $L^{p}(\mathbb{R}^{d};\mathbb{R}^{d})$ and 
\[
\lim_{\delta \to 0} \limsup_{n\to \infty} \|{\bf f}_{n} -  \mathbb{G}^\delta \ast {\bf f}_{n}\|_{L^{p}} = 0, 
\]
then  for any open and bounded subset $D$ of $\mathbb{R}^{d}$ the sequence $\{{\bf f}_{n}\}$ is relatively compact in $L^{p}(D; \mathbb{R}^{d})$. 
\end{lemma}
 
 \begin{proof}[Proof of Theorem  \ref{loc-compactness-S}]
From the assumption we have 
\begin{equation}\label{mainconditioncompact-appen}
\sup_{n\geq 1} \|{\bf u}_{n}\|_{L^{p}}^{p} + \sup_{n\geq 1}\int_{\mathbb{R}^{d}}\int_{\mathbb{R}^{d}}\rho(\bdx'-\bdx)\left|\mcD({\bf u}_n)(\bdx,\bdx')\right|^{p}d\bdx'd\bdx < \infty.
\end{equation}
Let $\Gamma^{\delta}({\bf z}) = {\rho({\bf z}) \over |{\bf z}|^{p}} \chi_{\complement B_{\delta}}({\bf z})$. Then for each $\delta$, $\Gamma^\delta \in L^1(\mathbb{R}^{d})$ and is radial, since $\rho$ is radial. Moreover, by assumption on $\rho$ \eqref{suff-loc-compactness}, $\|\Gamma^\delta\|_{L^1} \to \infty$ as $\delta \to 0. $ We next introduce the following sequence of integrable matrix functions 
\[
\mathbb{G}^{\delta}({\bf z}) = {d \,\Gamma^{\delta}({\bf z})\over \|\Gamma^\delta\|_{L^1} } {{\bf z}\otimes {\bf z} \over |{\bf z}|^{2}}.\]
Notice that since $\Gamma^\delta$ is radial, we have 
\[\int_{\mathbb{R}^{\delta}} {\Gamma^{\delta}({\bf z}) z_i^{2} \over |{\bf z}|^{2}}d{\bf z} =  { \|\Gamma^\delta\|_{L^1}\over d},\quad i=1,\cdots,d.\]
As a consequence $\{\mathbb{G}^{\delta}\}$ is an approximation to the identity.  
Now for each $n$ we have 
\[
\begin{aligned}
\| {\bf u}_{n} - \mathbb{G}^{\delta} \ast{\bf u}_n\|_{L^{p}}^{p} &= \int_{\mathbb{R}^{d}}\left|\int_{\mathbb{R}^{d}}\mathbb{G}^{d} (\bdy-\bdx)({\bf u}_n(\bdx) - {\bf u}_n(\bdy)) d\bdy \right|^{p} d\bdx\\
& \leq d^{p}  \int_{\mathbb{R}^{d}}\left|\int_{\mathbb{R}^{d}} {\Gamma^{\delta}({\bf z})\over \|\Gamma^\delta\|_{L^1} }\left|{\bdz\over |\bdz|}\cdot({\bf u}_n(\bdx) - {\bf u}_n(\bdz +\bdx)) \right|d\bdy \right|^{p} d\bdx\\
&\leq   d^{p}  \int_{\mathbb{R}^{d}}\int_{\mathbb{R}^{d}} \left|{\bdz\over |\bdz|}\cdot({\bf u}_n(\bdx) - {\bf u}_n(\bdz +\bdx)) \right|^{p} {\Gamma^{\delta}({\bf z})\over \|\Gamma^\delta\|_{L^1} } d\bdz d\bdx\\
&\leq {d^{p}\over \|\Gamma^\delta\|_{L^1} }  \int_{\mathbb{R}^{d}}\int_{\mathbb{R}^{d}} \left|{\bdz\over |\bdz|}\cdot({\bf u}_n(\bdx) - {\bf u}_n(\bdz +\bdx)) \right|^{p}  {\rho({\bf z}) \over |{\bf z}|^{p}}  d\bdz d\bdx\\
&= {d^{p}\over \|\Gamma^\delta\|_{L^1} }|{\bf u}_{n}|^p_{\mathcal{S}_{\rho,p}(\mathbb{R}^{d})}.
\end{aligned}
\]
By assumption on the sequence $\{{\bf u}_n\}$ \eqref{mainconditioncompact-appen}, we have  that for all $n$,  
\[
\| {\bf u}_{n} - \mathbb{G}^{\delta} \ast{\bf u}_n\|_{L^{p}}^{p} \leq C \,{d^{p}\over \|\Gamma^\delta\|_{L^1}}.  
\]
We take the limit as $\delta\to 0$ (uniformly in $n$) and use Lemma \ref{RFK-berzis} to conclude that ${\bf u}_n$ is compact in $L^{p}(\Omega;\mathbb{R}^{d}).$
\end{proof}

\end{document}